\documentclass[12pt]{article}
\usepackage[utf8]{inputenc}

\usepackage[utf8]{inputenc}

\usepackage[margin=1.5cm]{geometry}

\usepackage{mathtools}
\usepackage{amsfonts}
\usepackage{amssymb}
\usepackage{amsthm}
\usepackage{bbm}
\usepackage{enumitem}
\usepackage{accents}
\usepackage[hyphens]{url}
\usepackage{authblk}

\newtheorem{theo}{Theorem}
\newtheorem{lemma}[theo]{Lemma}
\newtheorem{cor}[theo]{Corollary}
\newtheorem{prop}[theo]{Proposition}

\theoremstyle{definition}\newtheorem{defn}[theo]{Definition}
\theoremstyle{remark}\newtheorem{remark}[theo]{Remark}

\numberwithin{theo}{section}
\numberwithin{equation}{section}

\newcommand{\R}{\mathbb{R}}
\newcommand{\N}{\mathbb{N}}
\newcommand{\C}{\mathbb{C}}

\newcommand{\Q}{\mathbb{Q}}
\newcommand{\BP}{\mathbf{P}}

\newcommand{\MCL}{\mathcal{L}}

\newcommand{\MCC}{\mathcal{C}}
\newcommand{\MCG}{\mathcal{G}}

\newcommand{\MCF}{\mathcal{F}}
\newcommand{\MCI}{\mathcal{I}}
\newcommand{\MCB}{\mathcal{B}}
\newcommand{\MCM}{\mathcal{M}}

\newcommand{\A}{\mathcal{A}}

\newcommand{\var}{\mathrm{var}}
\newcommand{\Var}{\mathrm{Var}}

\newcommand{\diff}{\mathop{}\!\mathrm{d}}

\newcommand{\htop}{h_{\text{top}}}

\newcommand{\GL}{\mathrm{GL}}
\newcommand{\SL}{\mathrm{SL}}

\newcommand{\innerproduct}[2]{\langle #1, #2 \rangle}

\newcommand{\Lt}{\MCL_t}
\newcommand{\rt}{\rho_t}

\newcommand{\uou}{\frac{u}{\|u\|} } 
\newcommand{\vov}{\frac{v}{\|v\|} } 
\newcommand{\Ptop}{P_{\mathrm{top}}}
\newcommand{\Pmeas}{P_{\mathrm{meas}}}

\allowdisplaybreaks

\title{On the superadditive pressure for 1-typical, one-step, matrix-cocycle potentials}
\author{Tom Rush}
\date{}

\begin{document}
\date{}
\maketitle
\begin{abstract}
Let \((\Sigma_T,\sigma)\) be a subshift of finite type with primitive adjacency matrix \(T\), \(\psi:\Sigma_T \rightarrow \R\) a H\"older continuous potential, and \(\A:\Sigma_T \rightarrow \GL_d(\R)\) a 1-typical, one-step cocycle. For \(t \in \R\) consider the sequences of potentials \(\Phi_t=(\varphi_{t,n})_{n \in \N}\) defined by
\[\varphi_{t,n}(x):=S_n \psi(x) + t\log \|\A^n(x)\|, \: \forall n \in \N.\]

Using the family of transfer operators defined in this setting by Park and Piraino, for all \(t<0\) sufficiently close to 0 we prove the existence of Gibbs-type measures for the superadditive sequences of potentials \(\Phi_t\). This extends the results of the well-understood subadditive case where \(t \geq 0\). Prior to this, Gibbs-type measures were only known to exist for \(t<0\) in the conformal, the reducible, the positive, or the dominated, planar settings, in which case they are Gibbs measures in the classical sense. We further prove that the topological pressure function \(t \mapsto \Ptop(\Phi_t,\sigma)\) is analytic in an open neighbourhood of 0 and has derivative given by the Lyapunov exponents of these Gibbs-type measures. 
\end{abstract}

\section{Introduction}\label{sec:intro}
Gibbs measures have been studied in dynamical systems since the work of Sinai \cite{Sin72} and, later, Bowen \cite{Bow75} in the 1970s. Roughly speaking, these are measures on subshifts of finite type for which cylinders have measure given by the Birkhoff sum of points in the cylinder. In recent years the theory has been generalised into the setting of random products of matrices, where a scalar times the logarithm of the norm of matrix products replaces the role of the Birkhoff sum. When the scalar is positive, this behaves subadditively. Thermodynamic formalism in the subadditive matrix setting is now very well understood: Gibbs-type measures were shown to exist under an irreducibility condition in \cite{FL02} and this was generalised in \cite{FK11} to matrices which are non-trivial, in the sense that the matrix products are not eventually identically equal to the 0 matrix.  

On the other hand, when the scalar is negative there is instead superadditive behaviour. Until now, Gibbs-type measures were only known to exist when the scalar is negative in the conformal setting, reducible settings (e.g. diagonal matrices), the positive setting \cite{FL02}, or the dominated, planar setting \cite{BR18,BKM20}. In each of these cases the measures are Gibbs measures for H\"older continuous potentials. The purpose of this paper is to construct Gibbs-type measures in the superadditive setting under a much more general condition on the matrices, namely the 1-typicality condition defined in \cite{BV04}. Our method is based on Bowen's proof, although there are significant differences. We use the transfer operators introduced in \cite{PPi22} which act on the space of functions on the product of the shift space and the projective space. Furstenberg-like measures on the projective space will play an important role.

For now and the rest of the paper, let \((\Sigma_T,\sigma)\) be a one-sided subshift of finite type on alphabet \(\{1,\ldots,q\}\) with primitive adjacency matrix \(T\). For details on the results in additive thermodynamic formalism for subshifts of finite type, we refer the reader to \cite{PPo90}.  A sequence of potentials \(\varphi_n:\Sigma_T \rightarrow \R\) is said to be \textit{subadditive} if each \(\varphi_n\) is continuous and 
\begin{equation}\label{eqn:subadditvedef}
    \varphi_{n+m}(x) \leq \varphi_n(x)+\varphi_m \circ \sigma^n(x), \: \forall n,m \in \N, \: \forall x \in \Sigma_T.
\end{equation}
A sequence of potentials \(\Phi=(\varphi_n)_{n \in \N}\) is \textit{superadditive} if \((-\varphi_n)_{n \in \N}\) is subadditive. Let us first discuss the results in subadditive thermodynamic formalism and the main difficulties that arise in the superadditive setting.

We define the set of admissible words of length \(n\) to be
\[\MCC_n := \{(x_0,\ldots,x_{n-1})\in \{1,\ldots,q\}^n: T_{i,i+1}=1, \forall 0\leq i \leq n-2\}\]
and we let \(\MCC_*:=\cup_{n \in \N} \MCC_n\). For an admissible word \(I=(i_0,\ldots, i_{n-1}) \in \MCC_*\), the \(n\)th level cylinder is defined to be 
\[[I]:=\{x \in \Sigma_T:x_j=i_j, \forall j \in \{0,\ldots, n-1\}\}. \]
Given a sequence of subadditive potentials \(\Phi=(\varphi_n)_{n \in \N}\), we can define the topological pressure 
\[\Ptop(\Phi,\sigma):=\lim_{n \rightarrow \infty}\frac{1}{n} \log \sum_{I \in \MCC_n} e^{\sup_{x \in [I]} \varphi_n(x)}.\]
The limit can be seen to exist by subadditivity and may be equal to \(-\infty\). 

Denoting by \(\MCM(\Sigma_T,\sigma)\) the set of \(\sigma\)-invariant probability measures on \(\Sigma_T\), for \(\mu \in \MCM(\Sigma_T,\sigma)\) we define
\[\Phi_*(\mu):=\lim_{n \rightarrow \infty} \frac{1}{n} \int \varphi_n \diff \mu.\]
The limit again exists by subadditivity,  may be \(-\infty\), and is equal to \(\inf_{n \in \N}  \frac{1}{n} \int \varphi_n \diff \mu\). The \textit{measure-theoretic pressure} is then defined to be 
\[P_{\text{meas}}(\Phi,\sigma)=\sup\{h(\mu,\sigma)+\Phi_*(\mu):\mu \in \MCM(\Sigma_T,\sigma) \},\]
where \(h(\mu,\sigma)\) is the entropy of \(\mu\) with respect to \(\sigma\). The main result of \cite{CFH08}, which was proved in a more general setting, says that
\[\Ptop(\Phi,\sigma)=P_{\text{meas}}(\Phi,\sigma).\]
This is known as a variational principle, and their work generalises the classical result proved for additive sequences; see \cite[\S 9]{Wal81}. 

A measure attaining the supremum in the measure-theoretic pressure is known as an \textit{equilibrium state} for \(\Phi\). Using the equality \(\Phi_*(\mu)=\inf_{n \in \N} \frac{1}{n} \int \varphi_n \diff \mu\), it is easy to show that \(\mu \mapsto \Phi_*(\mu)\) is upper semi-continuous in the weak* topology (see \cite[Proposition A.1]{FH10}). Hence, as the same also holds for the entropy map \(\mu \mapsto h(\mu,\sigma)\) \cite[Theorem 8.2]{Wal81}, it follows that there exists an equilibrium state for any subadditive sequence of potentials \(\Phi\). We say a measure \(\mu \in \MCM(\Sigma_T,\sigma)\) is a Gibbs-type measure for \(\Phi\) if there exists \(C_1,C_2>0\) such that for all \(n \in \N\) and \(x=(x_0,x_1,\ldots) \in \Sigma_T\)
\[ C_1 \leq \frac{\mu([x_0,\ldots,x_{n-1}])}{e^{-n \Ptop(\Phi,\sigma)+\varphi_n(x)} } \leq C_2.\]
When \(\Phi\) is additive, that is when there is equality in (\ref{eqn:subadditvedef}), this coincides with the classical definition of Gibbs measures defined in \cite{Bow75}. If a measure \(\mu\) is an ergodic Gibbs-type measure for a subadditive sequence of potentials \(\Phi\), then it is the unique equilibrium state for \(\Phi\) (see the proof of \cite[Theorem 5.5]{Fen11}).

We say a subadditive sequence \(\Phi=(\varphi_n)_{n \in \N}\) is \textit{locally constant} if each \(\varphi_n\) is locally constant on cylinders in \(\MCC_n\). For such sequences we will write \(\Phi(I)=\varphi_{n}(x)\) where \(I \in \MCC_n\) and \(x\) is any element of \([I]\). Ergodic Gibbs-type states are known to exist for locally constant subadditive sequences which are quasi-additive \cite[Theorem 5.5]{Fen11}. \textit{Quasi-additivity} means that there exists \(C \in \R\) and \(k \in \N\) such that for all \(I, J \in \MCC_*\) there is a \(K \in \cup_{0 \leq i \leq k} \MCC_i\) such that
\[\Phi(IKJ) \geq \Phi(I)+\Phi(J)+C.\]
It is easy to see that their result generalises to subadditive sequences with bounded variations (Definition \ref{defn:boundedvariations}).

Likewise to the subadditive case, for superadditive sequences of potentials one can define \(\Ptop(\Phi,\sigma)\) (with `\(\inf\)' replacing `\(\sup\)'), \(\Phi_*(\mu)\), \(\Pmeas(\Phi,\sigma)\), and Gibbs-type measures; see Section \ref{subsec:supadditivethermodynamicformalism} for details. Using a result proved in \cite{CPZ19}, we show in Proposition \ref{prop:supadditivevarprinciple} that the variational principle \(\Ptop(\Phi,\sigma)=\Pmeas(\Phi,\sigma)\) holds for any superadditive sequence of potentials \(\Phi\). This result was proved previously in some general matrix settings in \cite{Wu21}. Even with the variational principle, many of the techniques used in the subadditive case no longer work in
the superadditive case. For example, \(\mu \mapsto \Phi_*(\mu)\) is instead \textit{lower} semi-continuous which causes the standard proof of the existence of an equilibrium state to fail. Moreover, the proof of the existence of a Gibbs-type measure using quasi-additivity does not generalise to the superadditive case (there is no analogue of inequality (5.2) in \cite{Fen11}, for instance).

Let \(\A: \Sigma_T \rightarrow \GL_d(\R)\), \(d \in \N\), be a cocycle and let \(\psi:\Sigma_T \rightarrow \R\) be a H\"older continuous function.  For \(t \in \R\) we define the sequences of potentials \(\Phi_t=(\varphi_{t,n})_{n \in \N}\) by
\begin{equation}\label{eqn:Phitdefinition}
    \varphi_{t,n}:=S_n \psi(x)+t \log \|\A^n(x)\|,
\end{equation} 
where \(\A^n(x):=\A(\sigma^{n-1}x) \ldots \A(x)\), \(S_n \psi(x):=\sum_{i=0}^{n-1} \psi \circ \sigma^i(x)\) is the \(n\)th Birkhoff sum, and \(\|\cdot\|\) is the spectral norm. Notice that \(\Phi_t\) is subadditive when \(t>0\), additive when \(t=0\), and superadditive when \(t<0\). For \(\mu \in \MCM(\Sigma_T,\sigma)\) and \(1 \leq i \leq d\) we define the \(i\)th \textit{Lyapunov exponent} to be 
\begin{equation}\label{eqn:Lyapunovexponentdef}
    \lambda_i(\A,\mu):= \lim _{n \rightarrow \infty} \frac{1}{n} \int \log \sigma_i(\A^n(x)) \diff \mu,
\end{equation}
where \(\sigma_i(A)\) denotes the \(i\)th singular value of a matrix \(A\). The limits can be shown to exist by the fact that \(x \mapsto \log(\sigma_1(\A^n(x)) \ldots \sigma_i(\A^n(x))\) is subadditive for each \(1 \leq i \leq  d\). By the subadditive ergodic theorem, for any ergodic measure \(\mu \in \MCM(\Sigma_T,\sigma)\) and \(\mu\)-almost every \(x \in \Sigma_T\),
\[\lim_{n \rightarrow \infty} \frac{1}{n} \log \sigma_i(\A^n(x))=\lambda_i(\A,\mu).\]
By subadditivity, 
\[\lambda_1(\A,\mu)=\inf_{n \in \N} \frac{1}{n} \int \log \|\A^n(x)\| \diff \mu(x). \]
We also observe that
\[\Pmeas(\Phi_t,\sigma)=\sup_{\mu \in \MCM(\Sigma_T,\sigma)} \left\{h(\mu,\sigma)+\int \psi \diff \mu+t \lambda_1(\A,\mu) \right\}, \: \forall t \in \R.\] 

We refer to cocycles depending only on the first digit as \textit{one-step cocycles}. For one-step, full-shift cocycles, quasi-additivity of \(\Phi_t\) when \(t \geq 0\) was proved in \cite{Fen09} assuming only \textit{irreducibility}, that is that the matrices do not simultaneously fix a proper subspace. Hence, for such \(\Phi_t\) there exists an ergodic Gibbs-type measure which is the unique equilibrium state.

A setting which is particularly well understood is for irreducible, one-step, full-shift \(\GL_2(\R)\)-cocycles which satisfy the \textit{dominated condition}: a non-empty compact subset \(S\) of \(\GL_2(\R)\) is said to be dominated if there exists constants \(C>0\) and \(0<\tau<1\) such that
\begin{equation*}\label{eqn:domcondition}
    \frac{\sigma_2(A_1 \ldots A_n)}{\sigma_1(A_1 \ldots A_n)}<C \tau^n
\end{equation*}
for all \(A_1, \ldots A_n \in S\). 
Domination can also be formulated in terms of strongly invariant multi-cones; see \cite{BG09}. In this setting, existence of Gibbs-type measures for \(\Phi_t\) is known to hold for all \(t \in \R\). Remarkably, for each \(t \in \R\) the Gibbs-type measure for \(\Phi_t\) is in fact a Gibbs measure for a H\"older continuous potential \cite[Theorem 2.9]{BKM20}. However, this is not general. Theorem 2.9 in \cite{BKM20} further says that for irreducible, one-step, full-shift \(\GL_2(\R)\)-cocycles, this holds only when the matrices are dominated or \textit{strongly conformal}, meaning that the set of matrices \(\{\A^n(x): x \in \Sigma_T, n \in \N\}\) are all conformal with respect to the same conjugation matrix. 

Existence of Gibbs-type measures is intimately related to the smoothness of the pressure function \(t \mapsto \Ptop(\Phi_t,\sigma)\). In \cite{FL02} they prove that for one-step, full-shift cocycles taking values in the set of positive matrices, the pressure function \(t \mapsto \Ptop(\Phi_t,\sigma)\) is differentiable with derivative at \(t\) given by the Lyapunov exponent of the Gibbs-type measure for \(\Phi_t\). This further holds in the dominated planar-matrix setting by the results proved in \cite{BKM20}. For irreducible, one-step, full-shift cocycles, this is also known to hold when \(t>0\) \cite[Proposition 1.2]{FK11}. Moreover, in the strongly irreducible, proximal, i.i.d. setting, the analogous pressure function is known to be analytic on \(t>0\) when the distribution is compactly supported (see \cite[Theorem 8.8]{GP04} or \cite[Theorem A]{GP16}; we note their theorems considers distributions satisfying the more general finite exponential moment condition). 

\subsection{Statement of main theorems}

In this paper we will consider a 1-typical, one-step cocycle \(\A: \Sigma_T \rightarrow \GL_d(\R)\) (we defer the definition of 1-typicality until Section \ref{subsec:cocycles}). Under the more general fiber-bunching condition, 1-typicality was shown in \cite{BV04} to be a sufficient condition for the uniqueness of the top Lyapunov exponent of measures with continuous product structure, which includes Gibbs measures for H\"older continuous potentials \cite{Lep00}. They further showed that 1-typical, fiber-bunched cocycles form an open and dense subset in the set of fiber-bunched cocycles.  K. Park has also proved under these conditions that \(\Phi_t\) is quasi-additive when \(t > 0\) and hence admits a Gibbs-type measure which is the unique equilibrium state \cite{Par20} (see also \cite{BP21}). This setting has further been studied in related contexts in papers including \cite{AV06,KS13,BG19,GS19, CP21,Moh22,PPi22,CCZ23,Moh23} and \cite{Par23}. 

We now state our first theorem. The main novelty of our result is that we can prove the existence of Gibbs-type states for some superadditive matrix-cocycle potentials beyond the strongly conformal, reducible, and dominated settings.

\begin{theo}\label{theo:gibbsmeasuresexist}
    Let \((\Sigma_T,\sigma)\) be a subshift of finite type with primitive adjacency matrix \(T\), and let \(\psi:\Sigma_T \rightarrow \R\) be a H\"older continuous function. Also, let \(\A: \Sigma_T \rightarrow \GL_d(\R)\), \(d \in \N\), be a 1-typical, one-step cocycle. With \(\Phi_t\) as defined in (\ref{eqn:Phitdefinition}), for all \(t<0\) sufficiently close to 0 there exists an ergodic Gibbs-type measure \(\mu_t \in \MCM(\Sigma_T,\sigma)\) for \(\Phi_t\). Moreover, \(\mu_t\) is the unique equilibrium state for \(\Phi_t\).
\end{theo}


For \(t \geq 0\) we also denote by \(\mu_t\) the unique Gibbs-type measures for \(\Phi_t\) which are guaranteed by \cite{Par20}. Our second theorem again extends previous results to the superadditive case beyond the dominated, reducible, and strongly conformal settings. 


\begin{theo}\label{theo:Ptopanalytic}
In the setting of Theorem \ref{theo:gibbsmeasuresexist},  there exists an open neighbourhood of 0 such that the pressure function \(t \mapsto \Ptop(t):=\Ptop(\Phi_t,\sigma)\) is analytic and satisfies \(\Ptop'(t)=\lambda_1(\A,\mu_t)\). 
\end{theo}

It follows from the work of D\'iaz, Gelfert, and Rams \cite{DGR19,DGR22} that neither Theorem \ref{theo:gibbsmeasuresexist} nor Theorem \ref{theo:Ptopanalytic} can be extended to all of \(\R\) as there can be a phase transition at some point \(t<0\). We will discuss this in detail in Section \ref{sec:proofoftheocounter}.

\subsection{Key ideas}
We denote by \(\BP\) the \((d-1)\)-dimensional real projective space. In \cite{PPi22}, for 1-typical, fiber-bunched cocycles (which includes those that are one-step) they define the transfer operator 
\begin{equation*}\label{eqn:transferoperatorpsi}
    \MCL_t f(x,\overline{u})=\sum_{y \in \sigma^{-1} x} e^{ \psi(y)} \left\|\A(y)^*\uou \right\|^t f(y,\overline{\A(y)^*u}), \: t \in \R
\end{equation*}
acting on the space of \(\alpha\)-H\"older continuous functions on \(\Sigma_T \times \BP\) and prove that it is analytic for all \(\alpha>0\) sufficiently small (see \cite[Lemma 5.3]{PPi22}). This means that \(t \mapsto \int \MCL_t f \diff \nu\) is analytic for every Borel measure \(\nu\) on \(\Sigma_T \times \BP\) and every \(f \in C^{\alpha}(\Sigma_T \times \BP)\). In an open neighbourhood of 0, they further show that \(\MCL_t\) has a spectral gap and that the spectral radius is analytic \cite[Proposition 5.4]{PPi22}. 

Considering the classical setting \cite{Bow75}, a natural approach to prove Theorem \ref{theo:gibbsmeasuresexist} is to construct probability measures from the eigenmeasures and eigenfunctions given by the transfer operator. This defines measures, \(\nu_t\) say, on \(\Sigma_T \times \BP\). We can then project the measures \(\nu_t\) onto measures \(\mu_t\) on \(\Sigma_T\) and these are \(\sigma\)-invariant and ergodic (see \S \ref{sec:definingmut}). Moreover, \(\mu_0\) is the Gibbs measure corresponding to H\"older continuous potential \(\psi\). Note that we can also project the measures \(\nu_t\) onto \(\BP\) to define measures \(\eta_t\). In the case where \(t> 0\) and \(\Sigma_T\) is the full shift, to prove that the measure \(\mu_t\) is Gibbs-type requires proving that the projective measure \(\eta_t\) is not supported on a projective subspace. This argument was used previously in \cite[Proposition 3.4]{Pir20} in the strongly irreducible, proximal, full-shift, one-step setting. For subshifts of finite type and 1-typical, one-step cocycles, the argument holds with only minor modifications (see Lemma \ref{lem:mutgibbst>0}).  

When \(t<0\) the problem is more subtle. We define the dimension of a measure \(\eta\) on \(\BP\) to be equal to
\begin{equation}\label{eqn:dimensiondef}
    \dim \eta := \liminf_{r \rightarrow 0} \inf_{U \subset \BP} \frac{\log \eta(B(U,r))}{\log r},
\end{equation}
where the infimum is over all \((d-2)\)-dimensional projective subspaces (i.e. the projections of \((d-1)\)-dimensional subspaces of \(\R^d\)) and 
\[B(U,r):=\{\overline{u} \in \BP: \inf_{\overline{v} \in U} d_{\BP}(\overline{u},\overline{v})<r\}. \]
Here \(d_{\BP}\) is the natural metric on \(\BP\) (see (\ref{eqn:dpmetricdefinition})). Observe that \(\dim \eta\) is a lower bound for the lower local dimension of \(\eta\) at every point in \(\BP\). It turns out that, in contrast to the case when \(t > 0\), one can show that \(\mu_t\) is Gibbs-type when \(t<0\) by showing that \(\dim \eta_t>-t\). We remark that the property of not being supported on a projective subspace is well behaved under small perturbations, whereas the dimension of a measure may not be. Moreover, \(\dim \eta_t>0\) implies that \(\eta_t\) is not supported on a projective subspace, so the condition appearing for \(t<0\) is strictly stronger than that for \(t> 0\).

In \cite[Proposition 3.10]{PPi22} they prove that \((x,\overline{u})=(x,\overline{\xi_*}(x))\) for \(\nu_0\)-a.e. \((x,\overline{u}) \in \Sigma_T \times \BP\), where \(\overline{\xi_*}(x)\) is the element in \(\BP\) corresponding to the 0-dimensional slowest Oseledets' subspace of the inverse adjoint cocycle \(\A_*^{-1}\) (see Section \ref{subsec:cocycles}). When \(\overline{\xi_*}(x)\) is well defined and 0-dimensional, \(\overline{\xi_*}(x)\) is equivalently the density point of the transpose of \(\A^n(x)\); that is, \(\A^n(x)^*/\|\A^n(x)^*\|\) converges to a projection onto the 1-dimensional subspace of \(\R^d\) corresponding to \(\overline{\xi_*}(x)\). Thus, the measure \(\eta_0\) is a natural analogue of the Furstenberg measure. In the strongly irreducible, proximal, i.i.d. setting, it is proved in \cite[Theorem 14.1]{BQ16} that the dimension defined in (\ref{eqn:dimensiondef}) of the Furstenberg measure is strictly positive (see also \cite{Gui90}). The proof relies essentially entirely upon large deviation principles. By the LDPs recently proved for Gibbs measures under the 1-typicality condition assumed in this paper (see \cite[Theorem 1.5]{GS19} and \cite[Theorem B]{PPi22}), one would expect that the proof of Theorem 14.1 in \cite{BQ16} can be adapted to show that \(\dim \eta_0>0\). The main difficulty in proving Theorem \ref{theo:gibbsmeasuresexist} is extending this to a lower bound for \(\dim \eta_t\), which we require to be uniform for all \(t<0\) sufficiently close to 0.

Overcoming this difficulty requires two key elements. The first is a large deviation principle for the measures \(\nu_t\) which is proved using the analyticity of the transfer operator and its spectral radius. This allows us to show, amongst other things, that \((x,\overline{u})=(x,\overline{\xi_*}(x))\) for \(\nu_t\)-a.e. \((x,\overline{u})\), where \(\overline{\xi_*}(x)\) is as defined as before. Secondly, we show that there exists \(R_t\) with \(R_t \rightarrow 1\) as \(t \rightarrow 0\) and \(C_t>0\) such that \(\mu_t([I]) \leq C_t R_t^n \mu_0([I])\) for all \(I \in \MCC_n\) and \(n \in \N\). This allows us to prove large deviation estimates for the measures \(\mu_t\) from those recently proved for the Gibbs measure \(\mu_{0}\) (for the precise statement, see Lemma \ref{lem:LDPSformut}). Hence, combining these we are able to adapt the proof of \cite[Theorem 14.1]{BQ16} to get a lower bound for \(\dim \eta_t\), which we show can be taken to be uniform in \(t\) in a neighbourhood of 0. It will thus follow that \(\dim \nu_t>-t\) for all \(t<0\) sufficiently close to 0, and from this we can prove Theorem \ref{theo:gibbsmeasuresexist}. Theorem \ref{theo:Ptopanalytic} is then a straightforward consequence of our proof of Theorem \ref{theo:gibbsmeasuresexist} (which we prove for all \(t\) in an open neighbourhood of 0) and the analyticity of the spectral radius which was proved in \cite{PPi22}. In particular, we relate \(\Ptop(\Phi_t,\sigma)\) to the logarithm of the spectral radius of \(\Lt\) for those \(t\) for which we construct the Gibbs-type measures.

\subsection{Organisation}
The layout of the paper is as follows. The next section is devoted to introducing the main definitions and proving some preliminary lemmas. In Section \ref{sec:uniformldp} we modify the proof of Theorem B in \cite{PPi22} to prove a slightly stronger, uniform large deviation principle for Gibbs measures as we require this for our dimension lower bound. In Section \ref{sec:definingmut} we define the measures \(\nu_t\) using the operator \(\Lt\) and prove that their projections \(\mu_t\) onto \(\Sigma_T\) are ergodic, \(\sigma\)-invariant, and converge weak* to the Gibbs measure \(\mu_0\) as \(t \rightarrow 0\). Using an upper semi-continuity argument and the result proved in \cite{BV04} for Gibbs measures, we further show that they have unique top Lyapunov exponent for all \(t\) in a neighbourhood of 0. In Section \ref{sec:LDPfornut} we prove LDPs for the measures \(\nu_t\) and deduce from this that \((x,\overline{u})=(x,\overline{\xi_*}(x))\) for \(\nu_t\)-a.e. \((x,\overline{u})\). In Section \ref{sec:analysisofprojmeasures} we then show that \(\dim \eta_t>-t\) holds for all \(t<0\) sufficiently close to 0 and that this is sufficient to show that the measures \(\mu_t\) are Gibbs-type. We finish the proof of Theorem \ref{theo:gibbsmeasuresexist} in Section \ref{sec:proofoftheogibbs}, with most of the section devoted to showing uniqueness. Theorem \ref{theo:Ptopanalytic} is proved in Section \ref{sec:proofofanalytictheorem} by relating \(\Ptop(\Phi_t,\sigma)\) to the logarithm of the spectral radius of \(\Lt\) for all \(t\) in a neighbourhood of 0. In Section \ref{sec:convexity} we prove a theorem regarding strict convexity of the pressure function; in particular, we use Theorem \ref{theo:Ptopanalytic} to show that either \(t \mapsto \Ptop(\Phi_t,\sigma)\) is strictly convex in a neighbourhood of 0 or it is linear on all of \(\R\) (see Theorem \ref{theo:convexity}). We finish in Section \ref{sec:proofoftheocounter} by showing that Theorems \ref{theo:gibbsmeasuresexist} and \ref{theo:Ptopanalytic} cannot be extended to all \(t<0\), using the results proved in \cite{DGR19,DGR22}.

\section{Preliminaries}

\subsection{Subshifts of finite type}\label{subsec:subshiftsoffinitetype}
Throughout this paper \(T\) is a primitive \(q \times q \) matrix with entries in \(\{0,1\}\). We denote by \(\Sigma_T \subseteq \{1,\ldots, q\}^{\N_0}\) the one-sided subshift of finite type with adjacency matrix \(T\). We use \(\sigma\) to denote the left shift map on \(\Sigma_T\). For \(x \in \Sigma_T\) the cylinder of length \(n\) around \(x\) is defined by
\[[x]_n:=\{y \in \Sigma_T: y_i=x_i, \forall 0 \leq i \leq n-1\}.\]
In a similar way we define cylinders \([I]\) for \(I \in \MCC_*\), where recall \(\MCC_*\) is the set of all admissible words. Since \(T\) is primitive, \((\Sigma_T,\sigma)\) is \textit{topologically mixing} in the sense that there exists \(N \in \N\) such for all \(i,j \in \{1,\ldots,q\}\) and all \(n \geq N\) there exists \(K \in \MCC_n\) such that \([i,K,j]\) is non-empty.

We endow \(\Sigma_T\) with the metric \(d_{\Sigma_T}:\Sigma_T \times \Sigma_T \rightarrow [0,\infty)\) defined by 
\[d_{\Sigma_T}(x,y)=
\begin{cases}
1 & \text{ if } x_0 \not= y_0 \\
2^{-k} & \text{ if } x_i = y_i \text{ for all } i=0,\ldots,k-1 \text{ and } x_{k} \not= y_{k} \\
0 & \text{ if } x=y.
\end{cases}.\]

\subsection{Superadditive thermodynamic formalism}\label{subsec:supadditivethermodynamicformalism}
A sequence \(\Phi=(\varphi_n)_{n \in \N}\) of real-valued potentials on \(\Sigma_T\) is said to be \textit{superadditive} if each \(\varphi_n\) is continuous and
\[\varphi_{n+m}(x) \geq \varphi_n(x)+\varphi_m \circ \sigma^n(x), \: \forall n,m \in \N, \: \forall x \in \Sigma_T.\]
By a superadditivity argument, for each \(\mu \in \MCM(\Sigma_T,\sigma)\) we can define
\[\MCF_*(\Phi,\mu):=\lim_{n \rightarrow \infty} \frac{1}{n} \int \varphi_n \diff \mu=\sup_{n \in \N} \frac{1}{n}\int \varphi_n \diff \mu.\]
It is possible that \(\MCF_*(\Phi,\mu)=\infty\). For a superadditive sequence of potentials \(\Phi=(\varphi_n)_{n \in \N}\), we define the \textit{measure-theoretic pressure} by 
\[P_{\text{meas}}(\Phi,\sigma)=\sup\{h(\mu,\sigma)+\MCF_*(\mu):\mu \in \MCM(\Sigma_T,\sigma) \}\]
and the \textit{topological pressure} by
\begin{equation}\label{eqn:pressuredef}
    \Ptop(\Phi,\sigma):=\lim_{n \rightarrow \infty}\frac{1}{n} \log \sum_{I \in \MCC_n} e^{\inf_{x \in [I]} \varphi_n(x)}.
\end{equation} 

\begin{lemma}
    The limit defining the topological pressure in (\ref{eqn:pressuredef}) exists, but may be \(\infty\).
\end{lemma}

\begin{proof}
    We adapt the proof of \cite[Lemma 2.2]{Fen09}. Let \(k \in \N\) be such that for all \(I, J \in \MCC_*\) there exists \(K \in \MCC_k\) such that \(IKJ\) is admissible (such a \(k\) exists because \(T\) is primitive). For each pair \(I, J \in \MCC_*\) choose such a \(K_{I,J} \in \MCC_k\). Let
    \[s_n:=\sum_{I \in \MCC_{n}} e^{\inf_{x \in [I]} \varphi_{n}(x)}. \]
    For any \(n, m \in \N\) we have
    \begin{align*}
        s_{n+k+m}&=\sum_{(IKJ) \in \MCC_{n+k+m}} e^{\inf_{x \in [IKJ]} \varphi_{n+k+m}(x)}\\
 &\geq \sum_{I \in \MCC_{n}} \sum_{J \in \MCC_m} e^{\inf_{x \in [IK_{I,J}J]} \varphi_{n+k+m}(x)} \\
        &\geq \sum_{I \in \MCC_{n}} \sum_{J \in \MCC_m} e^{\inf_{x \in [IK_{I,J}J]} \varphi_{n}(x)+\varphi_k( \sigma^n(x))+ \varphi_m( \sigma^{n+k}(x))} \\
        &\geq  \sum_{I \in \MCC_{n}} \sum_{J \in \MCC_m} e^{\inf_{x \in [I]} \varphi_{n}(x)+\inf_{y \in [J]} \varphi_m (y)- \|\varphi_k\|_{\infty}} \\
        &=e^{-\|\varphi_k\|_{\infty}} \sum_{I \in \MCC_{n}} e^{\inf_{x \in [I]} \varphi_{n}(x)} \sum_{J \in \MCC_m} e^{\inf_{y \in [J]} \varphi_m(y)} \\
        &= e^{-\|\varphi_k\|_{\infty}} s_n s_m.
    \end{align*}
    Thus, the sequence 
    \[a_{n}:=e^{-\|\varphi_k\|_{\infty}} s_{n-k}\]
    is supermultiplicative in the sense that \(a_{n+m} \geq a_n a_m\) for all \(n,m> k\). It follows by superadditivity that the limit \(\lim_{n \rightarrow \infty} \frac{1}{n} \log s_n=\lim_{n \rightarrow \infty} \frac{1}{n} \log a_n\) exists.
\end{proof}

Given a continuous potential \(\varphi:\Sigma_T \rightarrow \R\) we define \(\Ptop(\varphi,\sigma):=\Ptop((S_n \varphi)_{n \in \N},\sigma)\). Explicitly,
\[\Ptop(\varphi,\sigma):= \lim_{n \rightarrow \infty}\frac{1}{n} \log \sum_{I \in \MCC_{n}} e^{\inf_{x \in [I]} \sum_{i=0}^{n-1} \varphi(\sigma^{i} x)}.\]
A standard argument shows that this is equivalent to the definition of topological pressure defined in \cite[\S 9]{Wal81} (i.e. using cylinders to define both separated and spanning sets). For \(n \in \N\) note that \((\Sigma_T,\sigma^n)\) is a topologically mixing subshift of finite type on alphabet \(\MCC_n\), so we likewise define
\[\Ptop(\varphi,\sigma^n):= \lim_{l \rightarrow \infty}\frac{1}{l} \log \sum_{I \in \MCC_{nl}} e^{\inf_{x \in [I]} \sum_{i=0}^{l-1} \varphi(\sigma^{in} x)},\]
which is again equivalent to the definition in \cite{Wal81}. By Proposition 2.1 in \cite{CPZ19}, for any superadditive sequence of potentials \(\Phi=(\varphi_n)_{n \in \N}\) we have
\begin{equation}\label{eqn:Pmeasprop}
    P_{\text{meas}} (\Phi,\sigma)=\lim_{n \rightarrow \infty} \Ptop\left(\frac{\varphi_n}{n},\sigma \right)=\lim_{n \rightarrow \infty} \frac{1}{n} \Ptop(\varphi_n,\sigma^n).
\end{equation}

The following proposition generalises Theorem 1.1 in \cite{Wu21} which was proved using similar methods. In particular, Wu showed that the variational principle holds for any
superadditive potentials \((t \log \|\A_n(x)\|)_{n \in \N}\) where \(t < 0\) and \(\A : \Sigma \rightarrow M_d(\R)\) is a full-shift, one-step matrix cocycle such that \(\A_n
(x) \not= 0\) for all \(x\) and \(n\). Wu further proved a second variational result (\cite[Theorem 1.2]{Wu21}) which allows \(\A_n(x) = 0\) for full-shift, one-step matrix cocycles \(\A\) satisfying the irreducibility condition. 

\begin{prop}[Variational principle]\label{prop:supadditivevarprinciple}
    Let \((\Sigma_T,\sigma)\) be a subshift of finite type with primitive adjacency matrix \(T\), and let \(\Phi=(\varphi_n)_{n \in \N}\) be a superadditive sequence of potentials on \(\Sigma_T\). Then,
\[\Ptop(\Phi,\sigma)=P_{\mathrm{meas}}(\Phi,\sigma).\]
\end{prop}
\begin{proof}
By (\ref{eqn:Pmeasprop}) it suffices to show that 
\[\Ptop(\Phi,\sigma)=\lim_{n \rightarrow \infty} \frac{1}{n} \Ptop(\varphi_n ,\sigma^n).\] 
Let \(k \in \N\) be such that for all \(I, J \in \MCC_*\) there exists \(K \in \MCC_k\) such that \(IKJ\) is admissible. For any \(n,l \in \N\) we have
\begin{align*}
    \frac{1}{n} \log \sum_{I \in \MCC_n} e^{\inf_{x \in [I]} \varphi_{n}(x)} &= \frac{1}{nl} \log \sum_{(I_0,\ldots, I_{l-1}) \in \MCC_{n}^l} e^{\sum_{i=0}^{l-1}\inf_{x \in [I_i]}\varphi_{n}(x)} \\
    &\leq \frac{1}{nl} \log \sum_{J \in \MCC_{(n+k)l}}  e^{ \inf_{x \in [J]}  \sum_{i=0}^{l-1} \varphi_{n}(\sigma^{i(n+k)} x)} \\
    &\leq \frac{1}{nl} \log \left( \sum_{J \in \MCC_{(n+k)l}} e^{ \inf_{x \in [J]} \sum_{i=0}^{l-1} \varphi_{n+k}(\sigma^{i(n+k)} x) }\right)+ \frac{1}{n}\log\|\varphi_k\|_\infty,
\end{align*}
where the last inequality follows by the inequality \(\varphi_{n+k} \geq \varphi_n+\varphi_k \circ \sigma^n \). Letting \(l \rightarrow \infty\) gives 
\[ \frac{1}{n} \log \sum_{I \in \MCC_n} e^{\inf_{x \in [I]} \varphi_{n}(x)} \leq \frac{1}{n} \Ptop(\varphi_{n+k},\sigma^{n+k})+\frac{1}{n}  \log\|\varphi_k\|_{\infty}.\]
Thus,
\[\Ptop(\Phi,\sigma) \leq \lim_{n \rightarrow \infty} \frac{1}{n} \Ptop(\varphi_n,\sigma^n).\]

On the other hand, for any \(n,l \in \N\), 
\begin{align*}
    \frac{1}{nl} \log \sum_{I \in \MCC_{nl}} e^{\inf_{x \in [I]} \varphi_{nl}(x)} &\geq \frac{1}{nl} \log \sum_{I \in \MCC_{nl}} e^{\inf_{x \in [I]} \sum_{i=0}^{l-1} \varphi_{n}(\sigma^{in} x)}.
\end{align*}
Hence, letting \(l \rightarrow \infty\) gives 
\[\Ptop(\Phi,\sigma) \geq \frac{1}{n} \Ptop(\varphi_n,\sigma^n),\]
so
\[\Ptop(\Phi,\sigma) \geq \lim_{n \rightarrow \infty} \frac{1}{n} \Ptop(\varphi_n,\sigma^n).\]
\end{proof}

We denote 
\[\var_n(\Phi):= \sup\{|\varphi_n(x)-\varphi_n(y)|: x,y \in \Sigma_T, y \in [x]_n\}.\]
\begin{defn}\label{defn:boundedvariations}
We say a sequence of continuous potentials \(\Phi=(\varphi_n)_{n \in \N}\) on \(\Sigma_T\) has \textit{bounded variations} if 
    \[\sup_{n \in \N} \var_n(\Phi)<\infty.\]
Equivalently, this holds if there exists \(C>0\) such that for all \(n \in \N\) and all \(x,y \in \Sigma_T\) with \(y \in [x]_n\),
\[C^{-1} \leq \frac{e^{\varphi_n(x)}}{e^{\varphi_n(y)}} \leq C.\]
\end{defn}

\begin{defn}\label{def:Gibbstype}
We say a measure \( \mu \in \MCM(\Sigma_T,\sigma)\) is a \textit{Gibbs-type measure} for a superadditive sequence \(\Phi=(\varphi_n)_{n \in \N}\) if there exists \(C_1,C_2>0\) and \(P \in \R\) such that for all \(x \in \Sigma_T\) and \(n \in \N\)
\[C_1 \leq \frac{\mu([x]_n)}{e^{-Pn+\varphi_n(x)}} \leq C_2.\]
\end{defn}
Note that for a Gibbs-type measure to exist \(\Phi\) must have bounded variations.

\begin{lemma}\label{lem:PequalsPtop}
     Let \(\Phi=(\varphi_n)_{n \in \N}\) be a superadditive sequence of continuous potentials on \(\Sigma_T\) with bounded variations, and suppose that \(\mu \in \MCM(\Sigma_T,\sigma)\) is a Gibbs-type measure for \(\Phi\). Then, the constant \(P\) is equal to \(\Ptop(\Phi,\sigma)\). Moreover, the convergence in \(\Ptop(\Phi,\sigma)\) is \(O(1/n)\).
\end{lemma}

\begin{proof}
    By the Gibbs-type property,
    \begin{align*}
        \frac{1}{n} \log \sum_{I \in \MCC_{n}} e^{\inf_{x \in [I]} \varphi_n(x)} &\geq \frac{1}{n} \log \sum_{I \in \MCC_n} C_1 \mu(I) e^{Pn} \\
        &=P+ \frac{1}{n} \log C_1.
    \end{align*}
    Similarly we have 
    \[\frac{1}{n} \log \sum_{I \in \MCC_{n}} e^{\inf_{x \in [I]} \varphi_n(x)} \leq P+\frac{1}{n} \log C_2. \]
    Hence, the lemma follows.
\end{proof}

\begin{lemma}\label{lem:gibbsstatesareequilib}
    Let \(\Phi=(\varphi_n)_{n \in \N}\) be a superadditive sequence of continuous potentials on \(\Sigma_T\) with bounded variations and suppose there exists a Gibbs-type measure \(\mu \in \MCM(\Sigma_T,\sigma)\) for \(\Phi\), then it is an equilibrium state for \(\Phi\).
\end{lemma}

\begin{proof}
    Let \(C_2>0\) be as given by Definition \ref{def:Gibbstype}. Note that for any \(n \in \N\),
    \begin{align*}
        \MCF_*(\mu) &\geq \frac{1}{n} \int \varphi_{n}(x) \diff \mu \\
        & \geq  \frac{1}{n} \sum_{I \in \MCC_n} \mu([I]) \inf_{x \in [I]} \varphi_n(x) \diff \mu.
    \end{align*}
    Hence, using the Gibbs-type property and Lemma \ref{lem:PequalsPtop}, we have
    \begin{align*}
        h(\mu,\sigma)+ \MCF_*(\mu)  &\geq  \limsup_{n \rightarrow \infty} \frac{1}{n} \sum_{I \in \MCC_n} \mu([I])(-\log \mu([I])+ \inf_{x \in [I]} \varphi_n(x)) \\
        &\geq  \limsup_{n \rightarrow \infty} \frac{1}{n} \sum_{I \in \MCC_n} \mu([I])(n\Ptop(\Phi,\sigma)-\log C_2) \\
        &= \Ptop(\Phi,\sigma).
    \end{align*}
\end{proof}

\begin{lemma}\label{lem:otherequil}
    Let \(\Phi=(\varphi_n)_{n \in \N}\) be a superadditive sequence of continuous potentials on \(\Sigma_T\) with bounded variations, and suppose there exists an ergodic Gibbs-type measure \(\mu \in \MCM(\Sigma_T,\sigma)\) for \(\Phi\). If \(\mu' \in \MCM(\Sigma_T,\sigma)\) is an ergodic equilibrium state for \(\Phi\) and \(\mu' \not=\mu\), then 
    \[\inf_{n \in \N}\inf_{x \in \Sigma_T} \frac{\mu'([x]_n)}{e^{-n \Ptop(\Phi,\sigma)+\varphi_{n}(x)}} =-\infty, \:\: \sup_{n \in \N}\sup_{x \in \Sigma_T} \frac{\mu'([x]_n)}{e^{-n \Ptop(\Phi,\sigma)+\varphi_{n}(x)}} =\infty.\]
\end{lemma}

\begin{proof}
    By the Gibbs-type property we can let \(C>0\) be such that for all \(x \in \MCC_n\), 
    \[C^{-1} \leq \frac{\mu([x]_n)}{e^{-n \Ptop(\Phi,\sigma)+\varphi_{n}(x)}} \leq C.\]
    By Lemma 5.4 in \cite{Fen11} we have that 
    \[\lim_{n \rightarrow \infty} \sum_{I \in \MCC_n} \mu([I]) \log \frac{\mu'([I])}{\mu([I])}=-\infty.\]
    Thus, for any \(M>0\) there exists \(n \in \N\) such that 
    \[ \sum_{I \in \MCC_n} \mu([I]) \log \frac{\mu'([I])}{\mu([I])}<-M.\]
    This implies that for some \(I \in \MCC_n\)
    \[\frac{\mu'([I])}{\mu([I])}<e^{-M},\]
    so for any \(x \in [I]\)
    \[\frac{\mu'([x]_n)}{e^{-n \Ptop(\Phi,\sigma)+\varphi_{n}(x)}} <C e^{-M}.\]
    Hence, the first equality follows. The other equality can be argued similarly since \cite[Lemma 5.4]{Fen11} also gives that
    \[\lim_{n \rightarrow \infty} \sum_{I \in \MCC_n} \mu'([I]) \log \frac{\mu([I])}{\mu'([I])}=-\infty.\]
\end{proof}

\subsection{\(g\)-functions}\label{subsec:gfunctions}
We now recall the definition of a \(g\)-function introduced by Keane \cite{Kea71} and later used by Ledrappier in \cite{Led74}. Their usefulness to us is twofold. Firstly, as in \cite{PPi22} we will work with a \(g\)-function cohomologous to \(\psi\) rather than work with \(\psi\) directly, both because it helps to simplify the presentation of some of the proofs and for ease of applying the lemmas and propositions proved in \cite{PPi22}. Secondly, and more importantly, we will use the main result of \cite{Led74} in our proof of the uniqueness statement in Theorem \ref{theo:gibbsmeasuresexist}. In particular, we will show that the Gibbs-type measures given by Theorem \ref{theo:gibbsmeasuresexist} are also equilibrium states for \(g\)-functions. We can then adapt the method of proving uniqueness used, for example, in \cite{BS03}.

\begin{defn}
    Let \(Y \subseteq \Sigma_T\) be a Borel set satisfying \(\sigma^{-1}(Y)=Y\). We say a measurable function \(g:\Sigma_T \rightarrow \R\) is a \textit{g-function} on \(Y\) if it is strictly positive on \(Y\), \(0\) on \(\Sigma_T \setminus Y\), and 
    \[\sum_{y \in \sigma^{-1} x} g(y) =1\]
    for all \(x \in Y\). We denote by \(\MCG(Y)\) the set of all \(g\)-functions on \(Y\).  
\end{defn}

For \(g \in \MCG(Y)\) we define the transfer operator 
\[L_{\log g} f(x):= \sum_{y \in \sigma^{-1} x} g(y) f(y)\]
acting on the space of measurable functions on \(\Sigma_T\). We have the following theorem from \cite{Led74} (see also \cite[Theorem 2.1]{Wal78}; we note that the proof in \cite{Wal78} does not make use of the continuity of the \(g\)-functions assumed in the paper, and this is not assumed in \cite{Led74}). 

\begin{theo}\label{theo:led}
    Let \(g \in \MCG(Y)\). The following are equivalent:
    \begin{enumerate}[label=(\roman*)]
        \item \(L_{\log g}^* \mu=\mu\).
        \item \(\mu\) is \(\sigma\)-invariant and for all \(f \in \MCL^1(\mu)\), \(\mathbb{E}_{\mu}(f|\sigma^{-1}\MCB)=\sum_{y \in \sigma^{-1}\sigma x} g(y) f(y)\) for \(\mu\)-almost all \(x\), where \(\MCB\) is the Borel \(\sigma\)-algebra. 
        \item \(\mu\) is a \(\sigma\)-invariant equilibrium state for \(\log g\); that is, \(h(\mu,\sigma)+\int \log g \diff \mu=\Pmeas(\log g, \sigma)=0\).
    \end{enumerate}
\end{theo}

\subsection{1-typical, one-step cocycles}\label{subsec:cocycles}
Let \(\A: \Sigma_T \rightarrow \GL_d(\R)\) be a map depending only the first digit. For each  \(n \in \N\) we set \[\A^n(x):=\A(\sigma^{n-1}x) \ldots \A(x).\]
This satisfies the cocycle relation
\[\A^{m+n}(x)=\A^m(\sigma^n x) \A^n(x), \: \forall x \in \Sigma_T,\: \forall n,m \in \N,\]
so defines a \textit{one-step cocycle} over \(\sigma\). We also define the \textit{adjoint inverse cocycle} \(\A_*^{-1}(x)\) by
\[\A_*^{-n}(x):=[\A^{n}(x)^{-1}]^*=[\A^{n}(x)^*]^{-1}, \]
where \(A^*\) denotes the adjoint (i.e. transpose) of a matrix \(A\). Again, \(\A_*^{-1}\) is a one-step cocycle over \(\sigma\).

We also define 
\[\A^{[n]}(x):=[\A^n(x)]^*=\A(x)^* \ldots \A(\sigma^{n-1}x)^*.\] Notice that \(\A^{[n]}(x)=\A_*^{-n}(x)^{-1}\); we will use this fact frequently in this paper. For \(n \in \N\) and \(I \in \MCC_n\) we will write \(\A^n(I):=\A^n(x)\) where \(x\) is any element of \([I]\), and similarly for \(\A_*^{-1}\) and \(\A^{[n]}\). 

We now briefly recall the definition of 1-typicality for one-step cocycles. For more information, including the more general definition for fiber-bunched cocycles, we refer the reader to \cite{BV04}.

\begin{defn}\label{defn:1typical}
    We say a one-step cocycle \(\A:  \Sigma_T \rightarrow \GL_d(\R)\) is \textit{1-typical} if there exists a periodic point \(p \in \Sigma_T\) of some period \(n\), a point \(z \in \Sigma_T\), and \(l \in \N\) such that \((p_{n-1},z_0)\) and \((z_{l-1},p_0)\) are admissible and 
    \begin{enumerate}[label={(\arabic*)}]
        \item \(P:=\A^n(p)\) has \(d\) distinct eigenvalues in absolute value, and
        \item for any \(I,J \subset \{1,\ldots, d\}\) with \(|I|+|J|\leq d\),
        \(\{\A^l(p)^{-1}\A^l(z) v_i:i \in I\} \cup \{v_j: j \in J\} \)
        are linearly independent, where \(\{v_i\}_{i=1}^d\) are the eigenvectors of \(P\). \label{item:1typicalitydeftwisting}
    \end{enumerate}
\end{defn}
One can think of 1-typicality as being an analogue of strong irreducibility and strong proximality for cocycles on subshifts of finite type. In fact, it is not difficult to show that one-step, full-shift \(\GL_2(\R)\)-cocycles which are proximal and strongly irreducible are 1-typical. 

\begin{defn}\label{def:xiwelldefined}
    For \(x \in \Sigma_T\), we say that \textit{the slowest Oseledets' subspace of the cocycle \(\A_*^{-1}\) at \(x\) is well defined}, or, more simply, \textit{\(\overline{\xi_*}(x)\) is well defined}, if 
\begin{enumerate}[label={(\arabic*)}]
    \item The limit \(\lim_{n \rightarrow \infty}  \frac{1}{n} \log \sigma_d(\A_*^{-n}(x))\) exists;
    \item There exists a subspace \(U \subseteq \R^d\) such that: \begin{enumerate} [label={(\roman*)}]
        \item for all \(u \in U \setminus\{0\}\), \(\lim_{n \rightarrow \infty} \frac{1}{n} \log \left\|\A_*^{-n}(x)\uou \right\|\) exists and is equal to \(\lim_{n \rightarrow \infty}  \frac{1}{n} \log \sigma_d(\A_*^{-n}(x))\),
        \item  for all \(u \notin U\), \(\liminf_{n \rightarrow \infty} \frac{1}{n} \log \left\|\A_*^{-n}(x)\uou \right\|>\lim_{n \rightarrow \infty}  \frac{1}{n} \log \sigma_d(\A_*^{-n}(x))\).
    \end{enumerate}
\end{enumerate}
When the above holds, we denote by \(\overline{\xi_*}(x)\) the projection of the subspace \(U\) onto \(\BP\). 
\end{defn}

By Oseledets' theorem \cite{Ose68}, \(\overline{\xi_*}(x)\) is well defined when \(x\) is a \(\mu\)-typical point of an ergodic measure \(\mu \in \MCM(\Sigma,\sigma)\); moreover,
\[\lim_{n \rightarrow \infty}  \frac{1}{n} \log \sigma_d(\A_*^{-n}(x))=\lambda_d(\A_*^{-1},\mu)=\lambda_1(\A,\mu)\]
for \(\mu\)-a.e. \(x \in \Sigma\). We remark that Oseledets' theorem applies since we can view \(\A_*^{-1}\) as a cocycle on the two-sided shift. 

It is easy to check that \(\overline{\xi_*}(\sigma x)\) is well defined whenever \(\overline{\xi_*}(x)\) is, has the same dimension, and satisfies \(\A_*^{-1}(x)\overline{\xi_*}(x)=\overline{\xi_*}(\sigma x)\). Similarly, for all \(x \in \Sigma_T\) such that \(\overline{\xi_*}(x)\) is well defined and any \(y \in \sigma^{-1} x\), \(\overline{\xi_*(y)}\) is well defined, has the same dimension, and satisfies \(\A_*^{-1}(y)\overline{\xi_*}(y)=\overline{\xi_*}(x)\). For \(x \in \Sigma_T\) such that \(\overline{\xi_*}(x)\)  is well defined and 0-dimensional (that is, the set \(U\) in Definition \ref{def:xiwelldefined} is 1-dimensional), we choose and fix \(\xi_*(x) \in \overline{\xi_*}(x)\) such that \(\|\xi_*(x)\|=1\). Note that, by the above discussion, the set 
\[Y:=\{x \in \Sigma_T:\overline{\xi_*}(x) \text{ is well defined and 0-dimensional} \} \]
satisfies \(\sigma^{-1}(Y)=Y\). In Section \ref{sec:proofoftheogibbs} we will use this fact in conjunction with Theorem \ref{theo:led} to prove the uniqueness statement in Theorem \ref{theo:gibbsmeasuresexist}.

\subsection{Complex analysis in Banach algebras}\label{subsec:complexanalysisbanach}
The theorems in \cite{PPi22} are proved using perturbation theory for analytic operators (see \cite{Kat95}). In this section we recall the definitions and results they use that we will need in this paper. For more details we refer the reader to \cite[\S6]{PPi22} and \cite[\S 3.14]{DS88}. Let \((X,\|\cdot\|)\) be a Banach algebra. 

\begin{defn}\label{def:analytic}
Suppose \(U \subseteq \C\) is open and \(f:U \rightarrow X\) is a function. We say that \(f\) is \textit{differentiable} at \(z_0\) if there exists an element \(f'(z_0) \in X\) such that 
\[\left\|\frac{f(z)-f(z_0)}{z-z_0} -f'(z_0) \right\| \xrightarrow[z \rightarrow z_0]{} 0.\]
We say \(f\) is \textit{analytic} on \(U\) if \(f\) is continuous and differentiable at each point in \(U\). 
\end{defn}

\begin{prop}{\cite[Theorem 6.2]{PPi22}}\label{prop:analyticchar}
Suppose \(U \subseteq \C\) is open and \(f:U \rightarrow X\). Then \(f\) is analytic if and only if \(\varphi(f(z)):U \rightarrow \C\) is analytic for all \(\varphi \in X^*\).
\end{prop}

For Banach algebras \(X_1,X_2\), let \(L(X_1,X_2)\) denote the space of bounded linear operators which map from from \(X_1\) to \(X_2\). The following corollary can be deduced from Proposition \ref{prop:analyticchar}.

\begin{cor}{\cite[Corollary 6.3]{PPi22}}\label{cor:analyticchar}
Suppose \(U \subseteq \C\) is open and let \(f: U \rightarrow L(X_1,X_2)\). Then \(f\) is analytic if and only if \(\langle x_2^*,  f(z) x_1 \rangle \) is analytic for all \(x_1 \in X_1\) and \(x_2^* \in X_2^*\). 
\end{cor}

We will also need:

\begin{lemma}{\cite[Lemma 6.5]{PPi22}}\label{lem:Lzxzisanalytic}
Let \(z \mapsto L_z\) be a function from \(U \subseteq \C\) to the space of bounded linear operators on \(X\), and let \(z \mapsto x_z\) be a function from \(U \) to \(X\).
\begin{enumerate}[label=(\arabic*)]
    \item Suppose that \(z \mapsto L_z\) is analytic, then \(z \mapsto L^*_z\) is analytic. 
    \item Suppose that \(z \mapsto x_z\) is analytic and  \(z \mapsto L_z\) is analytic, then \(z \mapsto L_z x_z\) is analytic. 
\end{enumerate}
\end{lemma}

\subsection{The transfer operator \(\Lt\)}\label{subsec:transferoperator}
Let \(\psi:\Sigma_T \rightarrow \R\) be a H\"older continuous function, \(\A:\Sigma_T \rightarrow \GL_d(\R)\) a 1-typical, one-step cocycle, and \(\Phi_t\) be as defined in (\ref{eqn:Phitdefinition}). We also let \(h_\psi\) be the positive-valued, H\"older continuous eigenfunction constructed in \cite{Bow75} corresponding to the spectral radius of the operator
\[ L_{\psi}f(x):=\sum_{y \in \sigma^{-1} x} e^{\psi(y)}f(y)\]
which is defined to be acting on a suitable space of H\"older continuous functions on \(\Sigma_T\). We define
\begin{equation*}\label{eqn:gfunction}
    g(x):=\frac{e^{\psi(x)}}{e^{\Ptop(\psi,\sigma)}} \cdot \frac{h_\psi(x)}{h_\psi(\sigma x)},
\end{equation*}
then \(g \in \MCG(\Sigma_T)\) is a H\"older continuous \(g\)-function cohomologous to \(\psi\). Denoting \(g^{(n)}(x) :=g(\sigma^{n-1}x)\ldots g(x)\), we have \( g^{(n)}(x) \asymp \mu_{\psi}([I]) \asymp e^{-n \Ptop(\psi,\sigma)+S_n \psi(x)}  \) for all \(I \in \MCC_n\) and \(x \in [I]\). Let \(\Phi^g_t:=(\varphi^g_{t,n})_{n \in \N}\) be defined by 
\[\varphi^g_{t,n}(x):=\log g^{(n)}(x) + t\log \|\A^n(x)\|, \: \forall n \in \N.\]
It follows that Gibbs-type states for \(\Phi^g_t\) and \(\Phi_t\) coincide and
\begin{equation}\label{eqn:PtopPhig}
\Ptop(\Phi^g_t, \sigma)=\Ptop(\Phi_t,\sigma)-\Ptop(\psi,\sigma).
\end{equation}
The next lemma is a consequence of \cite[Lemma 1.15]{Bow75}.

\begin{lemma}[Bounded distortion]\label{lem:boundeddistortion}
There exists \(C>0\) such that for all \(n \in \N\) and all \(x,y \in \Sigma_T\) with \(y \in [x]_n\), 
\[\frac{1}{C} \leq \frac{g^{(n)}(x)}{g^{(n)}(y)} \leq C.\]
\end{lemma}

Let \(d_{\BP}\) be the metric on \(\BP\) given for \(\overline{u}=\R u, \overline{v}=\R v \in \BP \) by 
\begin{equation}\label{eqn:dpmetricdefinition}
    d_{\BP}(\overline{u},\overline{v})=\frac{\|u \wedge v\|}{\|u\| \|v\|}.
\end{equation}
We endow \(\Sigma_T \times \BP\) with the metric
\[d_{\Sigma_T \times \BP}((x,\overline{u}),(y,\overline{v}))= \max\{ d_{\Sigma_T}(x,y),d_{\BP}(\overline{u},\overline{v})\}.\]
For \(t \in \R\) we define the transfer operator \(\Lt\) by 
\begin{equation*}\label{eqn:transferoperatorg}
    \Lt f(x,\overline{u})=\sum_{y \in \sigma^{-1} x} g(y) \left\|\A(y)^*\uou \right\|^t f(y,\overline{\A(y)^*u})
\end{equation*}
for all measurable \(f:\Sigma_T \times \BP \rightarrow \R\). Its iterations are given by
\begin{equation}\label{eqn:Ltiterated}
    \Lt^n f(x,\overline{u})=\sum_{y \in \sigma^{-n} x} g^{(n)}(y) \left\|\A^{[n]}(y)\uou \right\|^t f(y,\overline{\A^{[n]}(y)u}),
\end{equation}
where, recall, \(\A^{[n]}(x):=[\A^n(x)]^*=\A(x)^* \ldots \A(\sigma^{n-1}x)^*\).

For \(0 < \alpha \leq 1\), let \(C^{\alpha}(\Sigma_T \times \BP)\) denote the set of real-valued \(\alpha\)-H\"older continuous maps \(f:\Sigma_T \times \BP \rightarrow \R\). We endow \(C^{\alpha}(\Sigma_T \times \BP)\) with the usual norm: for \(f \in C^{\alpha}(\Sigma_T \times \BP)\) let
\[ |f|_{\alpha}:=\sup_{(x,\overline{u})\not=(y,\overline{v})} \frac{|f(x,\overline{u})-f(y,\overline{v})|}{d_{\Sigma_T \times \BP}((x,\overline{u}),(y,\overline{v}))^\alpha},\]
then the norm is defined by
\[\|f\|_{\alpha}:=|f|_{\alpha}+\|f\|_{\infty}.\]

\begin{prop}{\cite[Proposition 5.4]{PPi22}}\label{prop:LtintermsofPandS}
Suppose that \(\A\) is 1-typical and \(\alpha>0\) is sufficiently small, then there exists an open set \(U \subset \R\) containing 0 such that for any \(t \in U\) the operator \(\Lt:C^\alpha(\Sigma_T \times \BP) \rightarrow C^{\alpha}(\Sigma_T \times \BP)\) can be written as
\[\Lt=\rt(P_t+S_t),\]
where \(\rt \in \R\) and the operators \(P_t, S_t:C^\alpha(\Sigma_T \times \BP) \rightarrow C^{\alpha}(\Sigma_T \times \BP)\) satisfy \(P_t S_t =S_t P_t=0\). The functions \(t \mapsto \rt,P_t,S_t \) are all analytic on \(U\). Moreover, there exists constants \(C>0\) and \(0<\beta<1\) such that for all \(t \in U\) and all \(n \in \N\)
\[\|S_t^n\| \leq C\beta^n. \]
\end{prop}

\begin{remark}
In \cite[Proposition 5.4]{PPi22} they in fact show that there exists an open set \(U \subset \C\) containing 0 for which the above holds with operators instead acting on the space of complex-valued \(\alpha\)-H\"older continuous functions. For \(t \in \R\), \(\Lt\) maps real-valued functions to real-valued functions, hence the same must hold for \(P_t\) and \(S_t\). Thus, we can consider these operators to have domain \(U \subset \R\) and to be acting on the space of real-valued \(\alpha\)-H\"older continuous functions. Analyticity is then inherited from the fact that they are analytic in \(\C\). In particular, the results in \S \ref{subsec:complexanalysisbanach} all apply with \(U\) instead defined to be a subset of \(\R\). This is what we mean by analyticity in Proposition \ref{prop:LtintermsofPandS}. In this paper, \(t\) will always be real and functions will always be real valued.
\end{remark}

While not stated explicitly in the proposition, it is easy to see from the proof that \(\rt\) is the leading eigenvalue for \(\Lt\) and that \(\Lt\) has a spectral gap. Moreover, \(P_t\) is the projection onto the 1-dimensional subspace \(\ker (\rt I-\Lt)\) (see \cite[Lemma 4.4]{PPi22} and \cite[\S7.1.3]{Kat95}), so \(P_t^2=P_t\) and \(P_t \Lt=\Lt P_t=\rt P_t\). Letting \(h_t \in C^{\alpha}(\Sigma_T \times \BP)\) be an eigenfunction corresponding to \(\rt\), we have \(\Lt h_t=\rt h_t\) and \(P_t h_t=\rt h_t\). In Theorem 4.1 they further show that \(h_0(x,\overline{u}) \equiv 1\). Note also that \(\rho_0=1\).

We denote \(\MCM(\Sigma_T \times \BP)\) to be the set of Borel probability measures on \(\Sigma_T \times \BP\). Choose any \(m \in \MCM(\Sigma_T \times \BP)\), then \(\int 1 \diff m>0\). As \(t \mapsto P_t\) is analytic, \(\int P_t 1 \diff m>0\) for all \(t\) in a neighbourhood of 0. Since \(P_t\) is the projection onto \(\ker(\rt I -\Lt)\), for such \(t\) we can write \(P_t 1=\alpha_t h_t\) for some \(\alpha_t \in \R \setminus{0}\), so scaling \(h_t\) if necessary we may assume that \(P_t 1=h_t\). We define the measures \(\widetilde{\nu}_t \in \MCM(\Sigma_T \times \BP)\) by 
\begin{equation}\label{eqn:nutdef}
    \int f \diff \widetilde{\nu}_t=\left(\int h_t \diff m \right)^{-1} \int P_t f \diff m, \: \forall f \in C^{\alpha}(\Sigma_T \times \BP).
\end{equation}
Note that \(\widetilde{\nu}_t\) is well defined, as by an approximation argument it is sufficient to only specify the values of the integrals of the functions in \(C^{\alpha}(\Sigma_T \times \BP)\). For all \(f \in C^{\alpha}(\Sigma_T \times \BP)\),
\[\int \Lt f \diff \widetilde{\nu}_t 
    = \left(\int h_t \diff m \right)^{-1} \int P_t \Lt f \diff m 
    = \left(\int h_t \diff m \right)^{-1}\int \rt P_t f \diff m
    = \int f \diff \left(\rt \widetilde{\nu}_t \right).\]
Again by an approximation argument and by an application of Carath\'eodory's extension theorem, we have \(\Lt^* \widetilde{\nu}_t= \rt \widetilde{\nu}_t\). That is, \(\widetilde{\nu}_t\) is an eigenmeasure for \(\Lt^*\).  

Writing \(P_t f=\kappa_f h_t\) for \(f \in C^{\alpha}(\Sigma_T \times \BP)\), by (\ref{eqn:nutdef}) we have 
\[\int f \diff \widetilde{\nu}_t=\left(\int h_t \diff m \right)^{-1} \int P_t f \diff m=\left(\int h_t \diff m \right)^{-1} \int \kappa_f h_t \diff m= \kappa_f.\]
It follows that
\begin{equation}\label{lem:projectfontoh}
    P_t f=\left(\int f \diff \widetilde{\nu}_t \right) h_t.
\end{equation}
The following lemma is an immediate consequence of (\ref{lem:projectfontoh}) and Proposition \ref{prop:LtintermsofPandS}. 

\begin{lemma}\label{lem:iteratingLtdecay}
Let \(U \subset \R\), \(C>0\), and \(0<\beta<1\) be as in Proposition \ref{prop:LtintermsofPandS} (making \(U\) smaller if necessary so that \(\int P_t 1 \diff m>0\) on \(U\)). Then, for all \(f \in C^{\alpha}(\Sigma_T \times \BP)\),
\[\left\|\rt^{-n} \Lt^n f- \widetilde{\nu}_t(f) h_t \right\|_{\alpha} \leq C \|f\|_\alpha \beta^n. \]
\end{lemma}

\begin{lemma}\label{lem:htbound}
There exists an open set \(U \subset \R\) containing 0 such that for all \(t \in U\),
\[\inf_{t \in U} \inf_{(x,\overline{u}) \in \Sigma_T \times \BP} h_t(x,\overline{u})>0 \]
and 
\[\sup_{t \in U} \sup_{(x,\overline{u}) \in \Sigma_T \times \BP} h_t(x,\overline{u})<\infty .\]
\end{lemma}
\begin{proof}
Suppose the first inequality is not true, then there exists a sequence \(t_n \rightarrow 0\) and \((x_n,\overline{u_n}) \in \Sigma_T \times \BP\) such that \(h_{t_n}(x_n,\overline{u_n})<1/n\). By compactness, there exists a limit point \((x,\overline{u}) \in \Sigma_T \times \BP\). Without loss of generality we may assume that \((x_n,\overline{u_n}) \rightarrow (x,\overline{u})\). Let \(\delta_{(x,\overline{u})}\) be the Dirac measure on \((x,\overline{u})\). As \(t \mapsto h_t=P_t 1\) is analytic we have
\[h_t(x,\overline{u})= \int h_t \diff \delta_{(x,\overline{u})} \xrightarrow[t \rightarrow 0]{} \int h_0 \diff \delta_{(x,\overline{u})}=1.\]
It follows from this that \(\|h_{t_n}\|_{\alpha} \rightarrow \infty\). Thus, 
\[\|h_{t_n}-h_0\|_{\alpha}\xrightarrow[n \rightarrow \infty]{} \infty,\]
contradicting that \(h_t=P_t 1\) is analytic in an open neighbourhood of 0. The second inequality can be argued similarly.
\end{proof}

\section{A uniform LDP for Gibbs measures}\label{sec:uniformldp}
To prove Theorem \ref{theo:gibbsmeasuresexist} we require the following uniform LDP for Gibbs measures. This can be proved with only minor modifications to the proof of Theorem B in \cite{PPi22}, as we will show below. We state the proposition in the general setting of \cite{PPi22} (that is, for fiber-bunched, 1-typical cocycles), but in this paper we only require it holds for one-step, 1-typical cocycles.

\begin{prop}\label{theo:uniformLDPahat}
Let \((\Sigma_T,\sigma)\) be a (one or two-sided) subshift of finite type with primitive adjacency matrix \(T\). Let \(\psi:\Sigma_T \rightarrow \R\) be a H\"older continuous function and \(\mu_{\psi}\) the corresponding Gibbs measure. Also, let \(\A: \Sigma_T \rightarrow \GL_d(\R)\), \(d \in \N\), be a 1-typical, fiber-bunched cocycle. Then, for all \(\varepsilon>0\), 
\[\varlimsup_{n \rightarrow \infty}\frac{1}{n} \log \sup_{\overline{u}=\R u \in \BP} \mu_\psi \left \{x \in \Sigma_T: \left|\frac{1}{n}\log \left\|\A^{n}(x)\uou \right\|-\lambda_1(\A,\mu)\} \right|>\varepsilon \right\} < 0. \]
\end{prop}

To prove this we will need the following version of the Local Gärtner-Ellis Theorem. 

\begin{lemma}\label{lem:localgtheorem2}
    For each \(n \in \N\), let \(\{X_{i,n}\}_{i \in \MCI}\) be a set of random variables parameterised by some (possibly uncountable) index set \(\MCI\) which is independent of \(n\). Suppose that there exists \(\delta>0\) and a function \(c(t)\), differentiable at 0 with \(c(0)=c'(0)=0\), such that
    \begin{equation*}\label{eqn:uniformlocalg}
        \lim_{n \rightarrow \infty} \frac{1}{n} \log \sup_{i \in \MCI} \mathbbm{E}(e^{t X_{i,n}})=c(t)
    \end{equation*}
for all \(t \in [-\delta,\delta]\). Then, for all \(\varepsilon>0\),
\[\limsup_{n \rightarrow \infty} \frac{1}{n} \log \sup_{i \in \MCI} \mathbbm{P}(X_{i,n}>n\varepsilon) \leq \inf_{0\leq t \leq \delta} \{c(t)-t \varepsilon \}<0.\]
\end{lemma}

\begin{proof}
We adapt the proof of \cite[Chapter V, Lemma 6.2]{BL85}. By Markov's inequality we have for any \(i \in \MCI\) and \(t> 0\)
\[\mathbbm{P}(X_{i,n} \geq n\varepsilon)=\mathbbm{P}(e^{t X_{i,n} } \geq e^{t n\varepsilon})  \leq e^{-t n \varepsilon} \mathbbm{E}(e^{t X_{i,n} }), \]
so that
\[\limsup_{n \rightarrow \infty} \frac{1}{n} \log \sup_{i \in \MCI} \mathbbm{P}(X_{i,n} \geq n\varepsilon) \leq \inf_{0<t \leq \delta} \{c(t)-t \varepsilon \}.\]
It is clear that the right hand side is negative since \(c'(0)=0\). Since \(c(0)=0\) we have further have 
\[\inf_{0<t \leq \delta} \{c(t)-t \varepsilon \}=\inf_{0 \leq t \leq \delta} \{c(t)-t \varepsilon \}.\]
\end{proof}

\begin{remark}\label{remark:varnotnecessary}
In Theorem B and Corollary 1.2 in \cite{PPi22} they have the assumption that \(\Var=\rho''_0>0\). We remark that the corollary holds without this assumption, as this was only necessary for the lower bound of Theorem B in their paper (to see this, compare our Lemma \ref{lem:localgtheorem2} to their Lemma 5.9). In particular, as we only need an upper bound for the exponential decay rates in Propositions \ref{theo:uniformLDPahat} and \ref{prop:LDPfornut}, we do not need to assume that \(\Var=\rho''_0>0\) either.
\end{remark}

Recall the definition of \(\A^{[n]}\) from Section \ref{subsec:cocycles}.

\begin{lemma}\label{lem:ldpforA[n]}
For all \(\varepsilon>0\), 
\[\varlimsup_{n \rightarrow \infty}\frac{1}{n} \log \sup_{\overline{u} \in \BP} \mu_\psi \left \{x \in \Sigma_T: \left|\frac{1}{n}\log \left\|\A^{[n]}(x)\uou \right\|-\lambda_1(\A,\mu_\psi) \right|>\varepsilon \right\} < 0. \]
\end{lemma}

\begin{proof}
   Lemma 5.2.(3) in \cite{PPi22} says that for any \(\overline{u}=\R u \in \BP\)
   \[\int \left\|\A^{[n]}(x)\uou \right\|^t \diff \mu_{\psi}(x)=\int \Lt^n 1(x,\overline{u}) \diff \mu_{\psi}(x).\]
   So, for all \(\overline{u} \in \BP\) and all \(t\) in an open neighbourhood of 0,
    \[\int \left\|\A^{[n]}(x)\uou \right\|^t \diff \mu_{\psi}(x)=\rho_t^n \left(\int P_t 1(x,\overline{u})+S_t^n 1(x,\overline{u}) \diff \mu_{\psi}(x) \right).\]
    Thus, 
    \[\left| \frac{1}{n} \log \sup_{\overline{u} \in \BP} \int \left\|\A^{[n]}(x)\uou \right\|^t \diff \mu_{\psi}(x)- \log \rho_t \right| \leq \frac{1}{n} \log \left(  \sup_{\overline{u} \in \BP} \int P_t 1(x,\overline{u})+S_t^n 1(x,\overline{u}) \diff \mu_{\psi}(x) \right).\]
     It therefore follows from Proposition \ref{prop:LtintermsofPandS} that for all \(t\) in an open neighbourhood of 0 
    \[\lim_{n \rightarrow \infty} \frac{1}{n} \log \sup_{\overline{u} \in \BP} \int \|\A^{[n]}(x)u\|^t \diff \mu_{\psi}(x) =\log \rho_t.\]
    
    By \cite[Proposition 5.5]{PPi22},  \(t \mapsto \log \rt\) has derivative \(\lambda_1(\A,\mu_{\psi})\) at 0. Hence, applying Lemma \ref{lem:localgtheorem2} with 
    \[X_{\overline{u},n}:=\log \left\|e^{-n\lambda_1(\A,\mu_{\psi})} \A^{[n]}(x) \uou \right\| \]
    gives 
    \[\varlimsup_{n \rightarrow \infty}\frac{1}{n} \log \sup_{\overline{u} \in \BP} \mu_\psi \left\{x \in \Sigma_T: \frac{1}{n}\log \left\|\A^{[n]}(x)\uou \right\|>\lambda_1(\A,\mu_{\psi})+\varepsilon \right\}<0. \]
    Similarly, applying Lemma \ref{lem:localgtheorem2} with 
    \[X_{\overline{u},n}:=-\log \left\|e^{-n\lambda_1(\A,\mu_{\psi})} \A^{[n]}(x) \uou \right\| \]
    gives 
    \[\varlimsup_{n \rightarrow \infty}\frac{1}{n} \log \sup_{\overline{u} \in \BP} \mu_\psi \left\{x \in \Sigma_T: \frac{1}{n}\log \left\|\A^{[n]}(x)\uou \right\|<\lambda_1(\A,\mu_{\psi})-\varepsilon \right\}<0. \] 
    \end{proof}

Proposition \ref{theo:uniformLDPahat} now follows from Lemma \ref{lem:ldpforA[n]} in the same way that Theorem B follows from Theorem 5.8 in \cite{PPi22} (this is justified at the beginning of \cite[Section 5]{PPi22}). In particular, using Lemma \ref{lem:ldpforA[n]} one can prove that Proposition \ref{theo:uniformLDPahat} holds for the adjoint cocycle of \(\A\), and this is sufficient to prove Proposition \ref{theo:uniformLDPahat} because a cocycle is 1-typical if and only if its adjoint is (see \cite[Lemma 7.2]{BV04}). As the argument is the same, we omit the details.

\section{Defining \(\sigma\)-invariant measures \(\mu_t\)}\label{sec:definingmut}
Let \(\nu_t \in \MCM(\Sigma_T \times \BP)\) be the measures which are absolutely continuous with respect to the eigenmeasures \(\widetilde{\nu}_t\) with density \(h_t\). Considering the equilibrium states constructed in the classical setting \cite[\S 1]{Bow75}, natural candidates for the Gibbs-type measures in Theorem \ref{theo:gibbsmeasuresexist} are the measures \(\mu_t \in \MCM(\Sigma_T)\) defined by 
\[\int f(x) \diff \mu_t(x):=\int f(x) h_t(x,\overline{u}) \diff \widetilde{\nu}_t(x,\overline{u})=\int f(x) \diff \nu_t(x,\overline{u}), \:\: \forall f \in C^{\alpha}(\Sigma_T).\] 
Letting \(L_{\log g} \) be the operator defined by
\[L_{\log g} f (x):=\sum_{y \in \sigma^{-1} x} g(y) f(y),\]
we observe that \(\widetilde{\nu}_0\) must project to to unique eigenmeasure for \(L_{\log g}\) under the natural projection \(\pi:\Sigma_T \times \BP \rightarrow \Sigma_T\). Hence, \(\mu_0\) is the Gibbs measure corresponding to both \(\log g\) and \(\psi\); that is, \(\mu_0=\mu_{\psi}\). 

In this section we prove that the measures \(\mu_t\) are \(\sigma\)-invariant, ergodic and converge weak* to \(\mu_0\) as \(t \rightarrow 0\). Note that we can view each function in \(f \in C(\Sigma_T)\) as a function on \(\Sigma_T \times \BP\) by setting \(f(x,\overline{u}):=f(x)\). When there is no confusion, we will not distinguish between the two functions. 

\begin{lemma}\label{lem:sigmainvariant}
The measures \(\mu_t\) are \(\sigma\)-invariant.
\end{lemma}
\begin{proof}
We adapt the proof of Lemma 1.13 in \cite{Bow75}. Notice that for \(f \in C(\Sigma_T)\),
\begin{align*}
    ((\Lt h_t ) \cdot f)(x,\overline{u}) &= \sum_{y \in \sigma^{-1} x} g(y) \left\|\A(y)^* \uou \right\|^t h_t(y, \overline{\A(y)^* u}) f(x) \\
    &=  \sum_{y \in \sigma^{-1} x} g(y) \left\|\A(y)^* \uou \right\|^t h_t(y, \overline{\A(y)^* u}) f(\sigma y) \\
    &= \Lt(h_t \cdot (f \circ \sigma))(x,\overline{u}).
\end{align*}
Thus,
\begin{align*}
    \mu_t(f)&=\widetilde{\nu}_t(h_t f) \\
        &=\widetilde{\nu}_t (\rt^{-1} \Lt h_t \cdot f) \\
        &= \rt^{-1} \widetilde{\nu}_t ( \Lt(h_t \cdot (f \circ \sigma) ) ) \\
        &= \widetilde{\nu}_t(h_t \cdot (f \circ \sigma)) \\
        &= \mu_t(f \circ \sigma).
\end{align*}
\end{proof}

\begin{lemma}\label{lem:ergodic}
The measures \(\mu_t\) are ergodic. 
\end{lemma}

\begin{proof}
This follows similarly to the proof of Proposition 1.14 in \cite{Bow75}. We prove the stronger statement that \(\mu_t\) is mixing: this means that for any Borel sets \(E, F \subset \Sigma_T\),
\[\lim_{n \rightarrow \infty} \mu_t(E \cap \sigma^{-n} F) = \mu_t(E) \mu_t(F).\]
By a standard argument, we only need to check this holds for cylinders.

For any \(f_1, f_2 \in C(\Sigma_T)\) and \(n \in \N\),
\begin{align*}
    (\Lt^n (h_t f_1 ) \cdot f_2)(x,\overline{u}) &= \sum_{y \in \sigma^{-n} x} g(y) \left\|\A(y)^* \uou \right\|^t h_t(y, \overline{\A(y)^* u}) f_1(y) f_2(x) \\
    &=  \sum_{y \in \sigma^{-n} x} g(y) \left\|\A(y)^* \uou \right\|^t h_t(y, \overline{\A(y)^* u}) f_1(y) f_2(\sigma^n y) \\
    &= \Lt^n(h_t f_1 \cdot (f_2 \circ \sigma^n))(x,\overline{u}).
\end{align*}
Thus, for \(E=[x_0, \ldots, x_{k-1}]\) and \(F=[y_0, \ldots, y_{l-1}]\),
\begin{align*}
    \mu_t(E \cap \sigma^{-n} F) &= \mu_t(1_E \cdot 1_{\sigma^{-n}(F)}) \\
    &= \mu_t(1_E \cdot (1_F \circ \sigma^n)) \\
    &= \widetilde{\nu}_t(h_t 1_E \cdot (1_F \circ \sigma^n)) \\
    &= \rt^{-n} \Lt^{*n}\widetilde{\nu}_t (h_t  1_E \cdot (1_F \circ \sigma^n)) \\
    &= \widetilde{\nu}_t(\rt^{-n} \Lt^n(h_t 1_{E} \cdot  (1_F \circ \sigma^n)) \\
    &= \widetilde{\nu}_t( \rt^{-n} \Lt^n(h_t 1_{E}) \cdot 1_F).
\end{align*}
Therefore,
\begin{align*}
    |\mu_t(E) \cap \sigma^{-n}F)-\mu_t(E)\mu_t(F)|&= |\widetilde{\nu}_t( \rt^{-n} \Lt^n(h_t  1_{E}) \cdot 1_F)- \widetilde{\nu}_t(h_t 1_E) \widetilde{\nu}_t(h_t 1_F) | \\
    &=|\widetilde{\nu}_t \left((  \rt^{-n} \Lt^n(h_t  1_{E}) - \widetilde{\nu}_t(h_t 1_E) h_t) 1_F \right)| \\
    &\leq \|\rt^{-n} \Lt^n (h_t 1_E) -\widetilde{\nu}_t(h_t 1_E) h_t\|_\infty \widetilde{\nu}_t(1_F).
\end{align*}
Notice that \(h_t 1_E \in C^{\alpha}(\Sigma_T \times \BP)\), so by Proposition \ref{lem:iteratingLtdecay},
\[\|\rt^{-n} \Lt^n (h_t 1_E) -\widetilde{\nu}_t(h_t 1_E) h_t\|_\infty \leq \|\rt^{-n} \Lt^n (h_t 1_E) -\widetilde{\nu}_t(h_t 1_E) h_t\|_{\alpha} \rightarrow 0 \]
as \(n \rightarrow \infty\).
\end{proof}

\begin{lemma}\label{lem:mutweak*}
\(\mu_t \rightarrow \mu_0\) weak* as \(t \rightarrow 0\).
\end{lemma}

\begin{proof}
From \cite[Theorem 2.1] {Bi99} we have that for \(\mu_n, \mu \in \MCM(\Sigma_T)\), \(\mu_n \rightarrow \mu\) weak* if and only if
\(\mu(O) \leq \liminf_{n \rightarrow \infty} \mu_n(O)\) for every open set \(O \subset \Sigma_T\). We first show this holds for cylinders. Let \(C \subset \Sigma_T\) be a cylinder, then \(1_C \in C^{\alpha}(\Sigma_T \times \BP)\). By (\ref{eqn:nutdef}) and recalling that \(P_t 1=h_t\), we have
\begin{align*}
    \mu_{t}(E)&=\int 1_{C}(x) h_{t}(x,\overline{u}) \diff \widetilde{\nu}_t \\
    &= \frac{\int P_t(1_{C} h_{t}) \diff m }{\int h_t \diff m}.
\end{align*}
By Lemma \ref{lem:Lzxzisanalytic}, \(t \mapsto P_t(1_C h_t)\) is analytic. Moreover, \(h_t\) is analytic and bounded away from zero in an open neighbourhood of 0. It follows that for any \(t_n \in \R\) with \(t_n \rightarrow 0\),
\[\mu_{t_n}(C) \xrightarrow[n \rightarrow \infty]{} \frac{\int P_0(1_{C} h_{0}) \diff m }{\int h_0 \diff m} 
    = \mu_0(C). \]
Now observe that any open set in \(\Sigma_T\) can be written as a countable disjoint union of cylinders (here we are taking the convention that the empty set is a cylinder). Thus, for any open set \(E=\cup_{n \in \N} C_n\) and \(k \in \N\) we have 
\[\liminf_{n \rightarrow \infty} \mu_{t_n}(E) \geq \liminf_{n \rightarrow \infty} \mu_{t_n}(\cup_{i=1}^k C_k)= \mu_0(\cup_{i=1}^k C_k),\]
so taking the limit as \(k \rightarrow \infty\) finishes the proof.
\end{proof}

\subsection{Derivative of \(\log \rt\)}
We denote by \(\rt'\) the derivative of the map \(t \mapsto \rt\). In \cite[Proposition 5.5]{PPi22} they show that for any unit \(u \in \R^d\), 
\[\frac{\rho_0'}{\rho_0}= \lim_{n \rightarrow \infty} \frac{1}{n} \int  \log \|\A^{[n]}(x)u\| \diff \mu_0=\lambda_1(\A,{\mu_0}).\]
In the next section we will show that \(\rt'/\rt=\lambda_1(\A,\mu_t)\) for all \(t\) in an open neighbourhood of 0. For now we are only able to prove the following. 

\begin{lemma}\label{lem:rhotderivative}
For all \(n \in \N\),
\begin{align*}
    \frac{\rt'}{\rt}
    &= \int - \frac{1}{n} \log \left\|\A_*^{-n}(x) \uou \right\| h_t(x,\overline{u}) \diff \widetilde{\nu}_t \\
    &=  \int - \frac{1}{n} \log \left\|\A_*^{-n}(x) \uou \right\| \diff \nu_t.
\end{align*}
\end{lemma}

\begin{proof}
We adapt the proof of Proposition 5.7 in \cite{PPi22}. Differentiating the identity \(\MCL_t^n P_t=\rho_t^n P_t\), we obtain 
\[(\MCL_t^{n})' P_t+\Lt^n P_t' = n \rho_t^{n-1} \rho_t' P_t+\rho_t^n P_t'.\]
Since \(P_t \Lt^n =\rt^n P_t \) and \(P_t^2=P_t\), applying \(P_t\) to both sides gives
\[P_t(\MCL_t^{n})' P_t+\rt^n P_t P_t' = n \rho_t^{n-1} \rho_t' P_t+\rho_t^n P_t P_t'.\]
Cancelling \(\rt^n P_t P'_t\) from both sides further gives \(P_t(\MCL_t^{n})' P_t=n \rt^{n-1} \rt' P_t \). Recalling that \(P_t h_t=h_t\), \(\int h_t \diff \nu_t=1\), and \(P_t^* \widetilde{\nu_t}= \widetilde{\nu_t}\) (by (\ref{eqn:nutdef})), we have 
\[ \frac{\rt'}{\rho_t} = \frac{1}{n \rt^{n}} \int (\Lt^n)'h_t \diff \widetilde{\nu}_t.   \]
It follows from Lemma 5.3.(2) in \cite{PPi22} that 
\[(\Lt^n)' h_t= \sum_{y \in \sigma^{-n} x} g^{(n)}(y) \log \left\|\A^{[n]}(y) \frac{u}{\|u\|} \right\| \cdot \left\|\A^{[n]}(y) \uou \right\|^t h_t \left(y, \overline{\A^{[n]}(y)u } \right).\]
This shows that 
\begin{equation}\label{eqn:derivativelem}
    \frac{\rt'}{\rt} =\frac{1}{\rt^n} \int \sum_{y \in \sigma^{-n} x} g^{(n)}(y) \frac{1}{n} \log \left\|\A^{[n]}(y) \uou \right\| \cdot \left\|\A^{[n]}(y) \uou \right\|^t h_t \left(y, \overline{\A^{[n]}(y)u } \right) \diff \widetilde{\nu}_t(x,\overline{u}).
\end{equation}
Using the identity \(\A^{[n]}(x)^{-1}=\A_*^{-n}(x)\), the function
\[f(x,\overline{u}):=-\frac{1}{n}\log \left\|\A_*^{-n}(x) \uou \right\| \]
satisfies
\begin{equation*}
    f(y,\overline{\A^{[n]}(y)u}) = \frac{1}{n} \log \left\|\A^{[n]}(y) \uou \right\|
\end{equation*}
for all \((y,\overline{u}) \in \Sigma_T \times \BP\). Thus, by the formula for \(\Lt^n\) (see (\ref{eqn:Ltiterated})),
\begin{equation}\label{eqn:derivativelem2}
    \sum_{y \in \sigma^{-n} x} g^{(n)}(y) \frac{1}{n} \log \left\|\A^{[n]}(y) \uou \right\| \cdot \left\|\A^{[n]}(y) \uou \right\|^t h_t \left(y, \overline{\A^{[n]}(y)u } \right)=\Lt^n (f \cdot h_t)(x,\overline{u}). 
\end{equation}
Hence, combining (\ref{eqn:derivativelem}) and (\ref{eqn:derivativelem2}) and using that \((\Lt^{n})^* \widetilde{\nu}_t=\rt^n \widetilde{\nu}_t\), the lemma follows.
\end{proof}

\subsection{Lyapunov exponent gap}
\begin{lemma}\label{lem:lyapunovlowerbound}
\(\rt'/\rt \leq \lambda_1(\A,{\mu_t})\).
\end{lemma}
\begin{proof}
By Lemma \ref{lem:rhotderivative}, for any \(n \in \N\),
\begin{align*}
    \frac{\rt'}{\rt} &=\int - \frac{1}{n} \log \|\A_*^{-n}(x) u \| \diff \nu_t \\
    &\leq \int - \frac{1}{n} \log \sigma_d(\A_*^{-n}(x))  \diff \nu_t \\
    &= \int \frac{1}{n} \log \| \A^n(x)\| \diff \nu_t \\
    &=  \int \frac{1}{n} \log \| \A^n(x)\| \diff \mu_t.
\end{align*}
The lemma follows by letting \(n \rightarrow \infty\).
\end{proof}

\begin{lemma}\label{lem:lambda1geqlambda2}
For all \(t \in \R\) sufficiently close to 0, \(\rt'/\rt>\lambda_2(\A,\mu_t)\). 
\end{lemma}
\begin{proof}
Suppose for a contradiction that the lemma is not true. Then, we can find a sequence \(t_n<0\) such that \(t_n \rightarrow 0\) and \(\lambda_2(\A,\mu_{t_n}) \geq \rho_{t_n}'/\rho_{t_n}\) for all \(n \in \N\). Recall that \(\lambda_1(\A,\mu_0)=\rho_0'/\rho_0\). As \(\mu_0\) is a Gibbs measure, by \cite[Theorem 1]{BV04} we have \(\lambda_1(\A,\mu_0)>\lambda_2(\A,\mu_0).\) Moreover, by \cite[Proposition A.1]{FH10} we have that \(\mu \mapsto 
\lambda_1(\A,\mu)+\lambda_2(\A,\mu)\) is upper semi-continuous in the weak* topology. Hence, using Lemma \ref{lem:mutweak*}, it follows that 
\[2\lambda(\A,\mu_0)= 2 \lim_{n \rightarrow \infty} \rho_{t_n}'/\rho_{t_n} \leq \varlimsup_{n \rightarrow \infty} \lambda_1(\A,\mu_{t_n})+\lambda_2(\A,\mu_{t_n}) \leq \lambda_1(\A,\mu_0)+\lambda_2(\A,\mu_0)<2\lambda(\A,\mu_0),\]
which is a contradiction.
\end{proof}

\section{A large deviation principle for \(\nu_t\)}\label{sec:LDPfornut}
In this section we prove the following large deviation principle and deduce from this some properties of the measures \(\nu_t\).

\begin{prop}\label{prop:LDPfornut}
There exists \(\delta>0\) such that for all \(t \in [-\delta,\delta]\) and for all \(\varepsilon>0\),
\begin{equation*}
    \varlimsup_{n \rightarrow \infty}\frac{1}{n} \log \nu_t \left\{(x,\overline{u}) \in \Sigma_T\times \BP: -\frac{1}{n}\log \|\A_*^{-n}(x) u\| > \frac{\rt'}{\rt}+\varepsilon \right\} \leq \inf_{0 \leq s \leq \delta} \left\{\log \frac{\rho_{t+s}}{\rt}-s \frac{\rt'}{\rt}-s\varepsilon \right\}<0
\end{equation*}
and 
\begin{equation*}
    \varlimsup_{n \rightarrow \infty}\frac{1}{n} \log \nu_t \left\{(x,\overline{u}) \in \Sigma_T\times \BP: -\frac{1}{n}\log \|\A_*^{-n}(x) u\|<\frac{\rt'}{\rt}-\varepsilon \right\} \leq \inf_{0 \leq s \leq \delta} \left\{\log \frac{\rho_{t-s}}{\rt}+s \frac{\rt'}{\rt}-s\varepsilon \right\}<0.
\end{equation*}
\end{prop}

This has the corollary:
\begin{cor}\label{cor:LDPfornut}
There exists \(\delta>0\) such that for all \(\varepsilon>0\) there is \(\Lambda=\Lambda(\varepsilon)<0\) such that for all \(t \in [-\delta,\delta]\),
\begin{equation*}
    \varlimsup_{n \rightarrow \infty}\frac{1}{n} \log \nu_t \left\{(x,\overline{u}) \in \Sigma_T\times \BP: -\frac{1}{n}\log \|\A_*^{-n}(x) u\| \not\in \left(\frac{\rt'}{\rt}-\varepsilon, \frac{\rt'}{\rt}+\varepsilon \right) \right\} \leq \Lambda.
\end{equation*}
\end{cor}

\begin{proof}
For \(\varepsilon>0\) note that \((t,s) \mapsto \log \rho_{t+s}/\rt-s\rt'/\rt-s\varepsilon \) is continuous on \([-\delta,\delta] \times [0,\delta]\), so 
\[t \mapsto \inf_{0 \leq s \leq \delta} \{\log \rho_{t+s}/\rt-s\rt'/\rt-s\varepsilon  \} \]
is continuous. Since \(\inf_{0 \leq s \leq \delta}\{\log \rho_{t+s}/\rt-s\rt'/\rt-s\varepsilon  \} <0\) for all \(t \in [-\delta,\delta]\) and \([-\delta,\delta] \) is compact, we have
\[\sup_{t \in [-\delta,\delta]} \inf_{0 \leq s \leq \delta}\{\log \rho_{t+s}/\rt-s\rt'/\rt-s\varepsilon  \}<0. \]
Similarly we can show 
\[\sup_{t \in [-\delta,\delta]} \inf_{0 \leq s \leq \delta}\{\log \rho_{t-s}/\rt+s\rt'/\rt-s\varepsilon  \}<0. \]
\end{proof}

To prove Proposition \ref{prop:LDPfornut} we will use Lemma \ref{lem:localgtheorem2} (with \(\MCI\) being a singleton set). We first require a few lemmas.

\begin{lemma}\label{lem:Lt+s}
    For all \(f \in C(\Sigma_T \times \BP)\), all \(n \in \N\), and all \(t,s \in \R\),
    \[\MCL_{t}^n\left (f  \left\|\A_*^{-n}(x)\uou \right\|^{-s} \right)=\MCL^n_{t+s}(f).\]
\end{lemma}
\begin{proof}
By (\ref{eqn:Ltiterated}) and the identity \(\A_*^{-n}(x)=\A^{[n]}(x)^{-1}\),
\begin{align*}
    \MCL_{t}^n\left (f  \left\|\A_*^{-n}(x)\uou \right\|^{-s} \right)&=\sum_{y \in \sigma^{-n} x} g^{(n)}(y) \left\|\A^{[n]}(y)\uou \right\|^t f(y,\overline{\A^{[n]}(y)u}) \left\|\A_*^{-n}(y) \frac{\A^{[n]}(y)u}{\|\A^{[n]}(y)u\|}\right\|^{-s} \\
    &=\sum_{y \in \sigma^{-n} x} g^{(n)}(y) \left\|\A^{[n]}(y)\uou \right\|^{t+s} f(y,\overline{\A^{[n]}(y)u}) \\
    &= \MCL_{t+s}^n(f).
\end{align*}
\end{proof}

\begin{lemma}\label{lem:intfornutLDP}
For all \(t\) sufficiently close to 0 and all \(s \in \R\),
    \[\int \left\|\A_*^{-n}(x) \uou \right\|^{-s} h_t \diff \widetilde{\nu}_t=\frac{1}{\rt^n} \int \MCL_{t+s}^n h_t \diff \widetilde{\nu}_t.\]
\end{lemma}

\begin{proof}
    This follows immediately from Lemma \ref{lem:Lt+s} and the fact that \(\widetilde{\nu}_t\) is an eigenmeasure for \(\MCL_t\).
\end{proof}

Let \(\delta>0\) be small enough so that Proposition \ref{prop:analyticchar} holds on \((-3\delta,3\delta)\). Fix a \(t \in [-\delta,\delta]\). For each \(n \in \N\), let \(f_n:[-\delta, \delta] \rightarrow \R\) be defined by 
\begin{equation}\label{eqn:Fndef}
    f_n(s)= \int  \left\|\A_*^{-n}(x) \uou \right\|^{-s} h_t \diff \widetilde{\nu}_t.
\end{equation}

\begin{lemma}
For all \(s \in [-\delta,\delta]\),
\[\lim_{n \rightarrow \infty} \frac{1}{n} \log f_n(s)=\log \rho_{t+s}- \log \rt.\]
\end{lemma}

\begin{proof}
By Lemma \ref{lem:intfornutLDP} and Proposition \ref{prop:analyticchar},
\begin{align*}
    \frac{1}{n} \log f_n(s) 
    &=  \frac{1}{n} \log \int \MCL_{t+s}^n h_t \diff \widetilde{\nu}_t - \log \rt\\
    &= \frac{1}{n} \log \int \rho_{t+s}^n (P_{t+s}+S_{t+s}^n) h_t \diff \widetilde{\nu}_t - \log \rt\\
    &= \frac{1}{n} \log \left( \rho_{t+s}^n \left( \int h_{t+s} \diff \widetilde{\nu}_t \int h_t  \diff \widetilde{\nu}_{t+s}+ \int S_{t+s}^n h_t \diff \widetilde{\nu}_t \right) \right) - \log \rt \\
    & \xrightarrow[n \rightarrow \infty]{} \log \rho_{t+s} - \log \rt.
\end{align*}
\end{proof}

\begin{proof}[Proof of Proposition \ref{prop:LDPfornut}]
Fix \(t \in [-\delta,\delta]\). We apply Lemma \ref{lem:localgtheorem2} with \(X_n:=-\log \left\|e^{n\rt'/\rt} \A_*^{-n}(x) \uou \right\|\). We have for all \(s \in [-\delta,\delta]\),
\begin{align*}
    \lim_{n \rightarrow \infty} \frac{1}{n} \log \int \left\|e^{n\rt'/\rt} \A_*^{-n}(x) \uou \right\|^{-s}\diff \nu_t &= -s \frac{\rt'}{\rt}  +\lim_{n \rightarrow \infty} \frac{1}{n}  \log \int \left\| \A_*^{-n}(x) \uou \right\|^{-s} \diff \nu_t \\
    &=-s \frac{\rt'}{\rt} +\log \frac{\rho_{t+s}}{\rt}.
\end{align*}
This has derivative equal to 0 at \(s=0\). Thus, by Lemma \ref{lem:localgtheorem2},
\[\varlimsup_{n \rightarrow \infty}\frac{1}{n} \log \nu_t \left\{(x,\overline{u}) \in \Sigma_T\times \BP: -\log \left\|e^{n\rt'/\rt} \A_*^{-n}(x) \uou \right\|>n\varepsilon \right\} \leq \inf_{0 \leq s \leq \delta} \left\{\log \frac{\rho_{t+s}}{\rt}-s \frac{\rt'}{\rt}-s\varepsilon \right\} <0,\]
so
\[\varlimsup_{n \rightarrow \infty}\frac{1}{n} \log \nu_t \left\{(x,\overline{u}) \in \Sigma_T\times \BP: -\frac{1}{n}\log \left\| \A_*^{-n}(x) \uou \right\|> \frac{\rt'}{\rt}+\varepsilon \right\}\leq \inf_{0 \leq s \leq \delta} \left\{\log \frac{\rho_{t+s}}{\rt}-s \frac{\rt'}{\rt}-s\varepsilon \right\}.\]
Applying a similar argument with \(X_n:=\log \left\|e^{n\rt'/\rt} \A_*^{-n}(x) \uou \right\|\) gives
\[\varlimsup_{n \rightarrow \infty}\frac{1}{n} \log \nu_t \left\{(x,\overline{u}) \in \Sigma_T\times \BP: -\frac{1}{n}\log \left\| \A_*^{-n}(x) \uou \right\|<\frac{ \rt'}{\rt}-\varepsilon \right\} \leq  \inf_{0 \leq s \leq \delta} \left\{\log \frac{\rho_{t-s}}{\rt}+s \frac{\rt'}{\rt}-s\varepsilon \right\} <0.\]
\end{proof}

We are now able to prove the following.
\begin{prop}\label{prop:derivativeislyapunov}
For all \(t\) in an open neighbourhood of 0, \(\lambda_1(\A,\mu_t)=\rt'/\rt.\)
\end{prop}

\begin{proof}
By Lemma \ref{lem:lyapunovlowerbound} it suffices to show that \(\lambda_1(\A,\mu_t) \leq \rt'/\rt\) for all \(t\) in an open neighbourhood of 0. Suppose for a contradiction that \(\lambda_1(\A,\mu_t)>\rt'/\rt\). For \(\nu_t\)-a.e. \((x,\overline{u}) \in \Sigma_T \times \BP\) we have
\begin{equation}\label{eqn:showingderivativeislyapunov1}
    \lim_{n \rightarrow \infty} \frac{1}{n} \log \left\|\A_*^{-n}(x) \uou \right\| = -\frac{\rt'}{\rt}
\end{equation}
and
\begin{equation}\label{eqn:showingderivativeislyapunov2}
    \lim_{n \rightarrow \infty} \frac{1}{n}  \log \sigma_i(\A_*^{-n}(x)) = \lambda_i(\A_*^{-1},\mu_t)=-\lambda_{d-i+1}(\A,\mu_t), \: \: \forall i \in \{1,\ldots, d\}.
\end{equation}
Let \((x,\overline{u}) \in \Sigma_T\times \BP\) be \(\nu_t\)-typical. Since \(-\rt'/\rt>-\lambda_1(\A,\mu_t)=\lambda_d(\A_*^{-1}, \mu_t)\), (\ref{eqn:showingderivativeislyapunov1}) and (\ref{eqn:showingderivativeislyapunov2}) imply that \(\overline{u} \not=  \overline{\xi_*} (x)\), where recall \(\overline{\xi_*} (x)\) is the slowest Oseledets' subspace of \(\A_*^{-1}\) at \(x\). Thus, by Lemma \ref{lem:ergodic} and Oseledets' theorem \cite{Ose68}, 
\[\lim_{n \rightarrow \infty} \frac{1}{n} \log \left\|\A_*^{-n}(x)\uou \right\| \geq \lambda_{d-1}(\A_*^{-1},\mu_t)=-\lambda_2(\A,\mu_t), \]
so \(\lambda_2(\A,\mu_t) \geq \rt'/\rt\). However, this contradicts Lemma \ref{lem:lambda1geqlambda2} which says that \(\lambda_2(\A,\mu_t)<\rt'/\rt \) for all \(t\) in a neighbourhood of 0.
\end{proof}

In \cite[Proposition 3.10]{PPi22} they prove that \((x,\overline{u})=(x,\overline{\xi_*}(x))\) for \(\nu_0\)-a.e. \((x,\overline{u}) \in \Sigma_T \times \BP\). We can extend this to show that this further holds for \(t\) in an open neighbourhood of 0. 

\begin{prop}\label{prop:nutsupportonxi*}
For all \(t\) in an open neighbourhood of 0, \((x,\overline{u})=(x,\overline{\xi_*}(x))\) for \(\nu_t\)-a.e. \((x,\overline{u}) \in \Sigma_T \times \BP\).
\end{prop}

\begin{proof}
By Propositions \ref{prop:LDPfornut} and \ref{prop:derivativeislyapunov} and Kingman's subadditive ergodic theorem, we have that for \(\nu_t\)-a.e. \((x,\overline{u}) \in \Sigma_T \times \BP\),
\[\frac{1}{n} \log \left\|\A_*^{-n}(x) \uou \right\| \rightarrow -\lambda_1(\A,\mu_t)=\lambda_d(\A_*^{-1},\mu_t)\]
and 
\[\frac{1}{n} \log \sigma_i(\A_*^{-n}(x)) \rightarrow \lambda_i(\A_*^{-1},\mu_t), \:  \forall i\in \{1,\ldots, d\}.\]
Moreover, \(\lambda_d(\A_*^{-1},\mu_t)<\lambda_{d-1}(\A_*^{-1},\mu_t)\) by Proposition \ref{prop:derivativeislyapunov} and Lemma \ref{lem:lambda1geqlambda2}. Using again Lemma \ref{lem:ergodic} and Oseledets' theorem, it follows that \(\overline{u}= \overline{\xi_*}(x)\).
\end{proof}

Related to the cocycle \(\A_*^{-1}\) is the skew product \(F_{\A_*^{-1}}: \Sigma_T \times \BP \rightarrow \Sigma_T \times \BP\) defined by 
\[F_{\A_*^{-1}}(x,\overline{u})=\left(\sigma x, \overline{\A_*^{-1}(x) u}\right).\]
We remark that it is an immediate consequence of Proposition \ref{prop:nutsupportonxi*} that \(\nu_t\) is \(F_{\A_*^{-1}}\)-invariant for those \(t\) such that the proposition holds. In particular, this is true by \(\sigma\)-invariance of \(\mu_t\) and the fact that \(\A_*^{-1}(x)\overline{\xi_*}(x)=\overline{\xi_*}(\sigma x)\). Another immediate consequence of the proposition is that for all \(f \in C(\Sigma_T\times \BP)\) 
\begin{equation}\label{eqn:integratingnut}
    \int f(x,\overline{u}) \diff \nu_t(x,\overline{u})= \int f(x,\overline{\xi_*}(x)) \diff \mu_t(x).
\end{equation}

\begin{remark}
The transfer operators defined in this paper and in \cite{PPi22} are based on work by Le Page \cite{Pag82} which is in the strongly irreducible, proximal, i.i.d. setting under a finite exponential moment condition (see also \cite[Theorem 4.3]{BL85}). Accordingly, their family of transfer operators are instead defined to be acting on the space of continuous functions \(C(\BP)\). They likewise show that their family of transfers operators have a spectral gap and analytic spectral radius in a neighbourhood of 0.

Note that Proposition \ref{prop:LDPfornut} relies on the transfer operators being defined to be acting on the space of continuous functions on the product of the shift space and projective space. By instead using the transfer operators defined in \cite{WS20} combined with the methods in this paper, we expect that when \(\Sigma\) is the full shift and \(\psi\) depends only on the first digit (so that \(\mu_{\psi}\) is Bernoulli), it should be possible to prove the equivalent of Theorems \ref{theo:gibbsmeasuresexist} and \ref{theo:Ptopanalytic} assuming only proximality and strong irreducibility. 
\end{remark}

\section{Analysis of the projective measures \(\eta_t\)}\label{sec:analysisofprojmeasures}
\subsection{A sufficient condition for \(\mu_t\) to be Gibbs-type}
Recall that the metric \(d_{\BP}\) on \(\BP\) is given for \(\overline{u}=\R u, \overline{v}=\R v \in \BP \) by 
\[d_{\BP}(\overline{u},\overline{v})=\frac{\|u \wedge v\|}{\|u\| \|v\|}.\]
We also define
\[\delta_{\BP}(\overline{u},\overline{v})=\frac{|\innerproduct{u}{v}|}{\|u\| \|v\|},\]
where \(\innerproduct{\cdot}{\cdot}\) is the dot product. Denoting by \(v^\perp \subset \BP\) the projection of the subspace orthogonal to \(v\), this quantity can equivalently be written as
 \[\delta_{\BP}(\overline{u},\overline{v})=d_{\BP}(\overline{u},v^\perp):= \min_{\overline{w} \in v^\perp} d_{\BP}(\overline{u},\overline{w}).\]
We define the dimension of a measure \(\nu \in \MCM(\BP)\) to be
\[\dim \eta:=\sup \left\{s \geq 0: \sup_{\overline{v} \in \BP} \int \delta_{\BP}(\overline{u},\overline{v})^{-s} \diff \eta(\overline{u})<\infty \right\}. \]
It is straightforward to show that this is equivalent to the definition given in the introduction. It is this second characterisation that is more convenient to use in our proofs. 

For every element \(A \in \GL_d(\R)\) we can choose a composition \(A=K D L\), where \(D \in M_d(\R)\) is a diagonal matrix with positive entries \(D_{ii}\) in descending order and \(K,L \in O_d(\R)\). We denote by \(\overline{\upsilon_+}(A)\) the element in \(\BP\) corresponding to the largest singular value in the image of \(A\); that is,
\[\overline{\upsilon_+}(A):=\R K e_1.\]
Observe that for any \(u \in \overline{\upsilon_+}(A)\),
\begin{equation*}\label{eqn:weirdequality}
    \left\|A^* \uou\right\|=\|L^* D K^* K e_1\|=\|A^*\|.
\end{equation*}
We also denote by \(\gamma_{1,2}(A)\) the \textit{gap} of \(A\); that is, 
\[\gamma_{1,2}(A):=\frac{\sigma_2(A)}{\sigma_1(A)}.\]
Note that \(\gamma_{1,2}(A^*)=\gamma_{1,2}(A)\). We have the following lemma from \cite{BQ16} (Lemma 14.2).

\begin{lemma}\label{lem:BQlemma}
    For every \(A \in \GL_d(\R)\), \(\overline{u}=\R u, \overline{v}=\R v \in \BP\) one has 
    \begin{enumerate}[label=(\roman*)]
    \item \(\delta_{\BP}(\overline{\upsilon_+}(A^*),\overline{v}) \leq \frac{\|Av\|}{\|A\|\|v\|} \leq \delta_{\BP}(\overline{\upsilon_+}(A^*),\overline{v}) +\gamma_{1,2}(A)\)
    \item \(\delta_{\BP}(\overline{u},\overline{\upsilon_+}(A)) \leq \frac{\|A^* u\|}{\|A\|\|u\|} \leq \delta_{\BP}(\overline{u},\overline{\upsilon_+}(A)) +\gamma_{1,2}(A)\)
    \item \(d_{\BP}( \overline{A^* u}, \overline{\upsilon_+}(A^*)) \delta_{\BP}(\overline{u},\overline{\upsilon_+}(A)) \leq \gamma_{1,2}(A). \)
    \end{enumerate}
\end{lemma}

The following lemma will be used to show that the upper bound of the Gibbs-property holds for the measures \(\mu_t\).

\begin{lemma}\label{lem:intergralupperboundtoshowGibbs}
Let \(\eta \in \MCM(\BP)\) and suppose that \(0 \leq s < \dim \eta\). Then, there exists \(C>0\) such that for all \(A \in \GL_d(\R)\),
\[\int \left\|A^* \frac{u}{\|u\|} \right\|^{-s} \diff \eta(\overline{u}) \leq C \|A\|^{-s}. \]
\end{lemma}

\begin{proof}
Lemma \ref{lem:BQlemma}.\((ii)\) gives
\begin{align*}
    \int \left\|A^* \frac{u}{\|u\|} \right\|^{-s} \diff \eta(\overline{u}) &\leq  \left\|A \right\|^{-s} \int \delta_{\BP}(\overline{u},\overline{\upsilon_+}(A))^{-s} \diff \eta(\overline{u}) \\
    &\leq C \|A\|^{-s},
\end{align*}
where 
\[C:=\sup_{\overline{v} \in \BP} \int \delta_{\BP}(\overline{u},\overline{v})^{-s} \diff \eta(\overline{u})< \infty. \]
\end{proof}

\begin{lemma}\label{lem:mutJequality}
    For any \(I \in \MCC_n\)
    \[\mu_t([I])= \frac{1}{\rt^n}  \int g^{(n)}(Ix) \left\|\A^{[n]}(I)\uou \right\|^t 1_{\{x  \in \Sigma: (I,x_0) \textup{ is admissible}\}}(x) h_t(Ix,\overline{\A^{[n]}(I)u}) \diff \widetilde{\nu}_t(x,\overline{u}),\]
    where \(Ix\) denotes the element in \(\Sigma_T\) obtained by appending \(I\) onto the beginning of \(x\).
\end{lemma}
\begin{proof}
    Recalling that \(\widetilde{\nu}_t\) is the eigenmeasure for \(\Lt\), for \(I \in \MCC_n\) we have
\begin{align*}
    \mu_t([I])& = \int 1_{[I]}(x) h_t(x,\overline{u}) \diff \widetilde{\nu}_t(x,\overline{u}) \\
    &= \frac{1}{\rt^n} \int \sum_{y \in \sigma^{-n} x} g^{(n)}(y) \left\|\A^{[n]}(y)\uou \right\|^t 1_{[I]}(y) h_t(y,\overline{\A^{[n]}(y)u}) \diff \widetilde{\nu}_t(x,\overline{u})  \\
    &= \frac{1}{\rt^n}  \int g^{(n)}(Ix) \left\|\A^{[n]}(I)\uou \right\|^t 1_{\{x  \in \Sigma: (I,x_0) \text{ is admissible}\}}(x) h_t(Ix,\overline{\A^{[n]}(I)u}) \diff \widetilde{\nu}_t(x,\overline{u}).
\end{align*}
\end{proof}

To show the lower bound of the Gibbs-type property we will use that the measures \(\mu_t\) are fully supported, in particular that \(\min_{1 \leq i \leq q}\{\mu_t([i])\}>0\). For later use, we also prove an upper bound for the measures of cylinders as well.

\begin{lemma}\label{lem:upperboundformutJ}
    For all \(t\) in an open neighbourhood of 0, there exists \(C_t, R_t>0\), with \(R_t \rightarrow 1\) as \(t \rightarrow 0\), such that for all \(n \in \N\) and all \(I \in \Sigma_n\),
    \[C_t^{-1} R_t^{-n} \mu_0([I]) \leq \mu_t([I])\leq C_t R_t^n \mu_0([I])   \]
\end{lemma}

\begin{proof}
We show that this holds for \(t<0\); the proof for \(t>0\) can be argued similarly. Let \(C_{\mu_0}\) be the constant given by the Gibbs property for \(\mu_0\). By Lemma \ref{lem:mutJequality} and the fact that \(\left\|A \uou \right\| \geq \sigma_d(A)\) for every \(A \in \mathrm{GL}_d(\R)\) and \(u \in \R^d \setminus \{0\}\), we have 
\begin{align*} 
    \mu_t([I]) &= \frac{1}{\rt^n}  \int g^{(n)}(Ix) 
    \left\| \A^{[n]}(I) \uou \right\|^t 1_{\{x  \in \Sigma: (I,x_0) \text{ is admissible}\}}(x) h_t(Ix,\overline{\A^{[n]}(I)u}) \diff \widetilde{\nu}_t \\
    & \leq \frac{1}{\rt^n}  \int g^{(n)}(Ix) \sigma_d(\A^{[n]}(I))^t 1_{\{x  \in \Sigma: (I,x_0) \text{ is admissible}\}}(x) h_t(Ix,\overline{\A^{[n]}(I)u}) \diff \widetilde{\nu}_t \\
    &\leq \frac{1}{\rt^n} \sup(h_t)  \left(\min_{x \in \Sigma_T}\{\sigma_d(\A(x))\}\right)^{tn} \int g^{(n)}(Ix) \diff \widetilde{\nu}_t  \\
     &\leq \frac{1}{\rt^n}  C_{\mu_0} \sup(h_t) \left(\min_{x \in \Sigma_T}\{\sigma_d(\A(x))\}\right)^{tn} \mu_{0}([I]),
\end{align*}
where the last inequality follows since \(\mu_0\) is the Gibbs measure for \(\log g\).

For the other direction,  let \(k \in \N\) be such that for all \(I, J \in \MCC_*\) there exists \(K \in \MCC_k\) such that \(IKJ\) is admissible. Fix \(n \in \N\) and \(I \in \MCC_n\). Note that for each \(x \in \Sigma_T\) there is at least one \(z \in \sigma^{-(n+k)}(x)\) with \(z \in [I]\). For all \(\overline{u} \in \BP\) we have
\begin{align*}
    \Lt^{n+k} 1_{[I]}(x,\overline{u})&=\sum_{y \in \sigma^{-(n+k)} x} g^{(n+k)}(y) \left\|\A^{[n]}(y)\uou \right\|^t 1_{[I]}(y) h_t(y,\overline{\A^{[n]}(y)u}) \\
    &\geq   g^{(n+k)}(z) \left\|\A^{[n]}(z)\uou \right\|^t h_t(z,\overline{\A^{[n]}(z)u}) \\
    &\geq C_{\mu_0}^{-1} e^{-k\|\log g\|_{\infty}} \inf(h_t) \left(\max_{x \in \Sigma_T} \|\A(x) \| \right)^{tn} \mu_0([I]).
\end{align*}
Hence,
\[\mu_t([I]) \geq  \frac{1}{\rt^{n+k}} C_{\mu_0}^{-1} e^{-k\|\log g\|_{\infty}} \inf(h_t) \left(\max_{x \in \Sigma_T} \|\A(x) \| \right)^{tn} \mu_0([I]).\]
Thus, we can set \(C_t=C_{\mu_0} \max\{\inf(h_t)^{-1} e^{k\|\log g\|_{\infty}} \rt^k,\sup(h_t)\}\) and
\[R_t:= \max\left \{ \rt^{-1} \left(\min_{x \in \Sigma_T}\{\sigma_d(\A(x))\}\right)^{t},  \rt \left(\max_{x \in \Sigma_T} \|\A(x)\| \right)^{-t} \right\}.\]
Since \(\rt \rightarrow \rho_0= 1\) as \(t \rightarrow 0\), it is clear that \(R_t \rightarrow 1\) as \(t \uparrow 0\).
\end{proof}

Recall that \(\eta_t \in \MCM(\BP)\) is the projection of \(\nu_t\) onto \(\BP\).

\begin{lemma}\label{lem:sufficientconditionformut}
    If \(t<0\) and \(\dim \eta_t>-t\), then \(\mu_t\) is a Gibbs-type measure for \(\Phi_t\) and \(\Ptop(\Phi_t, \sigma)=\log \rt+ \Ptop(\psi,\sigma)\).
\end{lemma}

\begin{proof}
By Lemmas \ref{lem:boundeddistortion}, \ref{lem:htbound}, and  \ref{lem:mutJequality}, we have that for any \(n \in \N\), \(I \in \MCC_n\), and \(y \in [I]\),
\begin{align*}
    \frac{\mu_t([I])}{\rt^{-n}}
    &=\int g^{(n)}(Ix) \left\|\A^{[n]}(I)\uou \right\|^t  1_{\{ x: (I,x_0) \text{ admissible}\}}(x) h_t(Ix, \overline{\A^{[n]}(I)u}) \diff \widetilde{\nu}_t(x,\overline{u}) \\
    &\leq C_{b.d.} \sup(h_t) \inf(h_t)^{-1}  g^{(n)}(y)  \int\left\|\A^{[n]}(y)\uou \right\|^t h_t(x,\overline{u}) \diff \widetilde{\nu}_t(x,\overline{u}) \\
    &= C_{b.d.} \sup(h_t)\inf(h_t)^{-1}  g^{(n)}(y)  \int\left\|\A^{[n]}(y)\uou \right\|^t \diff \eta_t(\overline{u}) 
\end{align*}
where \(C_{b.d.}\) is the constant given by bounded distortion (Lemma \ref{lem:boundeddistortion}). Hence, applying Lemma \ref{lem:intergralupperboundtoshowGibbs}, for some \(C_2=C_2(t)>0\) we have
\begin{align*}
    \frac{\mu_t([I])}{\rt^{-n}} &\leq C_2 g^{(n)}(y) \|\A^{n}(y)\|^t.
\end{align*}

For the other direction we can instead use the trivial inequality \(\left\|A \uou \right\| \leq \|A\|\) for any \(A \in \GL_d(\R)\) and \(u \in \R^d\). Using again Lemmas \ref{lem:boundeddistortion}, \ref{lem:htbound}, and \ref{lem:mutJequality}, this gives that for some \(C=C(t)>0\) and any \(n \in \N\), \(I \in \MCC_n\) and \(y \in [I]\),
\begin{align}
    \frac{\mu_t([I])}{\rt^{-n}} &= \int g^{(n)}(Ix) \left\|\A^{[n]}(I)\uou \right\|^t  1_{\{ x: (I,x_0) \text{ admissible}\}}(x) h_t(Ix, \overline{\A^{[n]}(I)u}) \diff \widetilde{\nu}_t(x,\overline{u}) \nonumber\\
    &\geq C g^{(n)}(y) \|\A^{[n]}(y)\|^t \int 1_{\{ x: (I,x_0) \text{ admissible}\}}(x) \diff \nu_t(x,\overline{u}) \nonumber \\
    &\geq C g^{(n)}(y) \|\A^{[n]}(y)\|^t \min\{\mu_t([i]): 1 \leq i \leq q \} \nonumber \\
    &= C g^{(n)}(y) \|\A^{n}(y)\|^t \min\{\mu_t([i]): 1 \leq i \leq q \}. \nonumber
\end{align}
By Lemma \ref{lem:upperboundformutJ} and the fact the Gibbs measure \(\mu_0\) has full support, we have
\[\min\{\mu_t([i]): 1 \leq i \leq q\}>0.\]
We define \(C_1=C_1(t)>0\) to be 
\[C_1:=C \cdot  \min\{\mu_t([i]): 1 \leq i \leq q \}>0.\]

Thus, we have shown that for all \(n \in \N\), \(I \in \MCC_n\), and \(y \in [I]\),
\[C_1 \leq \frac{\mu_t([I])}{\rt^{-n} g^{(n)}(y) \|\A^{n}(y)\|^t} \leq C_2.\]
Hence, \(\mu_t\) is the Gibbs-type measure for \(\Phi_t^g\) and consequently for \(\Phi_t\). By Lemma \ref{lem:PequalsPtop} we have that \(\Ptop(\Phi_t^g)=\log \rho_t\), so by (\ref{eqn:PtopPhig}) we further have that \(\Ptop(\Phi_t,\sigma)=\log \rt+\Ptop(\psi,\sigma)\). This concludes the proof of Lemma \ref{lem:sufficientconditionformut}.
\end{proof}

\subsection{Lower bound for \(\dim \eta_t\)}
We now adapt the argument in \cite{BQ16} to prove the following proposition. The existence of the Gibbs-type measures in Theorem \ref{theo:gibbsmeasuresexist} will be a straightforward consequence of this combined with Lemma \ref{lem:sufficientconditionformut}.

\begin{prop}\label{prop:dimensionbound}
    There exists \(\delta>0\) such that \(\inf_{t \in [-\delta,\delta]} \dim \eta_t >0\).
\end{prop}

To ease the notation, we will denote \(\lambda_1(\A,\mu_0)\) by \(\lambda_1\) and \(\lambda_2(\A,\mu_0)\) by \(\lambda_2\). 

\begin{lemma}\label{lem:LDPSformut}
    For all \(\varepsilon>0\) there exists \(\delta>0\) and \(\Lambda<0\) such that for all \(t \in [-\delta,\delta]\):
    \begin{enumerate}[label=(\roman*)]
    \item \(\varlimsup_{n \rightarrow \infty}\frac{1}{n} \log \mu_t\{x \in \Sigma_T: |\frac{1}{n}\log \|\A^{[n]}(x) \xi_*( \sigma^n x) \|-\lambda_1|>\varepsilon\} \leq \Lambda\),
    \item \(\varlimsup_{n \rightarrow \infty}\frac{1}{n} \log \mu_t\{x \in \Sigma_T: |\frac{1}{n}\log \|\A^{[n]}(x)\|-\lambda_1|>\varepsilon\} \leq \Lambda\),
    \item \(\varlimsup_{n \rightarrow \infty}\frac{1}{n} \log \mu_t\{x \in \Sigma_T: |\frac{1}{n}\log \gamma_{1,2}(\A^{[n]}(x))-(\lambda_2-\lambda_1 )|>\varepsilon\} \leq \Lambda\),
    \item \(\varlimsup_{n \rightarrow \infty}\frac{1}{n} \log \sup_{\overline{v} \in \BP}\mu_t \left\{x \in \Sigma_T: \left|\frac{1}{n}\log \left\|\A^{n}(x)\frac{v}{\|v\|} \right\|-\lambda_1 \right|>\varepsilon \right\} \leq \Lambda\).
    \end{enumerate}
\end{lemma}

We will show \((i)\) is a consequence of Corollary \ref{cor:LDPfornut}. Parts \((ii)\), \((iii)\), and \((iv)\) will follow from the LDPs which have been proved for the Gibbs measure \(\mu_0\) and the upper bounds in Lemma \ref{lem:upperboundformutJ}.

\begin{proof}
Let us first prove \((i)\). By Corollary \ref{cor:LDPfornut} and Propositions \ref{prop:derivativeislyapunov} and \ref{prop:nutsupportonxi*}, for any \(\varepsilon>0\) there exists \(\delta>0\) and \(\Lambda<0 \) such that for all \(t \in [-\delta,\delta]\),
\[ \varlimsup_{n \rightarrow \infty} \mu_t \left\{x \in \Sigma_T: \left|-\frac{1}{n} \log \| \A_*^{-n}(x) \xi_*(x)\|-\lambda_1(\A,\mu_t)\right|>\frac{\varepsilon}{2} \right\} \leq \Lambda.\]
Making \(\delta\) smaller if necessary, we may assume that \(|\lambda_1(\A,\mu_t)-\lambda_1|<\frac{\varepsilon}{2}\). Thus, for all \(t \in [-\delta,\delta]\),
\[ \varlimsup_{n \rightarrow \infty} \mu_t \left\{x \in \Sigma_T: \left|-\frac{1}{n} \log \| \A_*^{-n}(x) \xi_*(x)\|-\lambda_1\right|>\varepsilon \right\} \leq \Lambda.\]
Statement \((i)\) then follows simply by noting that 
\begin{equation}\label{eqn:inverserelation}
    \|\A^{[n]}(x) \xi_*(\sigma^ n x)\|=\left\| \A^{[n]}(x) \frac{\A_*^{-n}(x)\xi_*(x)}{\|\A_*^{-n}(x)\xi_*(x)\|} \right\|=\|\A_*^{-n}(x) \xi_*(x)\|^{-1}.
\end{equation}

We now demonstrate the proof of \((iii)\); part \((ii)\) is similar and hence is omitted. For \(\varepsilon>0\) and \(n \in \N\), let 
\[\MCB_n(\varepsilon):= \left\{x \in \Sigma_T: \left|\frac{1}{n}\log \gamma_{1,2}(\A^{[n]}(x))-(\lambda_2-\lambda_1 ) \right|>\varepsilon \right\}.\]
Applying the large deviation principle \cite[Theorem 1.5]{GS19} to both \(\A\) and the exterior product cocycle \(\wedge^2 \A\) (see \cite[\S 2.1]{GS19}) gives
\[\Lambda_0:=\varlimsup_{n \rightarrow \infty}\frac{1}{n} \log \mu_0(\MCB_n(\varepsilon))<0. \]
We can choose \(\delta>0\) small enough so that 
\[\Lambda:=\sup_{t \in [-\delta,\delta]} \left\{\log R_t+\Lambda_0 \right\}<0.\] 
By Lemma \ref{lem:upperboundformutJ}, it therefore follows that for all \(t \in [-\delta,\delta]\)
\[\varlimsup_{n \rightarrow \infty}\frac{1}{n} \log \mu_t(\MCB_n(\varepsilon)) \leq \Lambda. \]
This proves \((iii)\).

Part \((iv)\) follows similarly. For \(\overline{v} \in \BP\), \(\varepsilon>0\), and \(n \in \N\), let
\[\MCB_n(\overline{v},\varepsilon):= \left\{x \in \Sigma_T: \left|\frac{1}{n}\log \left\|\A^{n}(x)\vov \right\|-\lambda_1 \right|>\varepsilon \right\}.\]
By Proposition \ref{theo:uniformLDPahat} one has 
\[\Lambda_0:=\varlimsup_{n \rightarrow \infty}\frac{1}{n} \log \sup_{\overline{v} \in \BP} \mu_0(\MCB_n(\overline{v},\varepsilon))<0. \]
Hence, with \(\Lambda\) and \(\delta\) as before,
\[\varlimsup_{n \rightarrow \infty}\frac{1}{n} \log \sup_{\overline{v} \in \BP} \mu_t(\MCB_n(\overline{v},\varepsilon)) \leq \Lambda. \]
\end{proof}

Proposition \ref{prop:dimensionbound} will be proved using part (iii) of the following lemma. Parts (i) and (ii) are only needed as intermediate steps towards proving (iii). 

\begin{lemma}\label{lem:xi*doesntaccumulateonsubspaces}
    For all \(\varepsilon>0\) there exists \(\delta>0\) and \(\Lambda<0\) such that for all \(t \in [-\delta,\delta]\): 
    \begin{enumerate}[label=(\roman*)]
        \item \(\varlimsup_{n \rightarrow \infty} \frac{1}{n} \log \mu_t \{x \in \Sigma_T: d_{\BP}(\overline{\xi_*}(x),\overline{\upsilon_+}(\A^{[n]}(x)))> e^{-(\lambda_1-\lambda_2-\varepsilon)n}  \}  \leq \Lambda \),
        \item \(\varlimsup_{n \rightarrow \infty} \frac{1}{n} \log \sup_{\overline{v} \in \BP} \mu_t \{x \in \Sigma_T: \delta_{\BP}(\overline{\upsilon_+}(\A^{[n]}(x)),\overline{v})< e^{-\frac{\varepsilon}{2} n}  \} \leq \Lambda \), 
        \item \(\varlimsup_{n \rightarrow \infty} \frac{1}{n} \log \sup_{\overline{v} \in \BP} \mu_t \{x \in \Sigma_T: \delta_{\BP}(\overline{\xi_*}(x),\overline{v})< e^{-\varepsilon n}  \}  \leq \Lambda \).
    \end{enumerate}
\end{lemma}

\begin{proof}
We adapt the proof of \cite[Proposition 14.3]{BQ16}. We may assume \(0<\varepsilon<\frac{1}{2}(\lambda_1-\lambda_2)\). We let \(\delta>0\) and \(0>\Lambda>\Lambda'\), where \(\delta\) and \(\Lambda'\) are as given by Lemma \ref{lem:LDPSformut}. Let \(t \in [-\delta,\delta]\). By Lemma \ref{lem:LDPSformut} there exists \(N \in \N\) such that for all \(n \geq N\) and all \(\overline{v} \in \BP\) there exists a subset \(\Omega_n=\Omega_n(t,\overline{v},\varepsilon) \subset \Sigma_T\) with \(\mu_t(\Omega_n)>1-e^{\Lambda n}\) such that for all \(x \in \Omega_n\) the quantities
\[\left|\frac{1}{n} \log \|\A^{[n]}(x)\|-\lambda_1 \right|,\]
\[\left|\frac{1}{n} \log \left\|\A^{n}(x)\vov \right\|-\lambda_1 \right|,\]
\[\left |\frac{1}{n}\log \|\A^{[n]}(x) \xi_*( \sigma^n x) \|-\lambda_1 \right|,\]
\[\left|\frac{1}{n}\log \gamma_{1,2}(\A^{[n]}(x))-(\lambda_2-\lambda_1) \right|\]
are all bounded by \(\varepsilon/8\). 

We check that, provided \(N\) is chosen large enough, for any \(x \in \Omega_n\) one has 
\[d_{\BP}(\overline{\xi_*}(x),\overline{\upsilon_+}(\A^{[n]}(x))) \leq e^{-(\lambda_1-\lambda_2-\varepsilon)n},\]
\[\delta_{\BP}(\overline{\upsilon_+}(\A^{[n]}(x)),\overline{v}) \geq e^{-\frac{\varepsilon}{2} n},\]
\[\delta_{\BP}(\overline{\xi_*}(x),\overline{v}) \geq e^{-\varepsilon n}.\]
For the first, using parts (ii) and (iii) of Lemma \ref{lem:BQlemma}, 
\begin{align*}
    d_{\BP}(\overline{\xi_*}(x),\overline{\upsilon_+}(\A^{[n]}(x)) &= d_{\BP}(\A^{[n]}(x) \overline{\xi_*}(\sigma^n x),\overline{\upsilon_+}(\A^{[n]}(x))) \\
    &\leq \frac{\gamma_{1,2}(\A^{[n]}(x))}{\delta_{\BP}(\overline{\xi_*}(\sigma^n x),\overline{\upsilon_+}(\A^n(x)))} \\
    &\leq \frac{\gamma_{1,2}(\A^{[n]}(x))}{\frac{\|\A^{[n]}(x)\xi_*(\sigma^n x) \|}{\|\A^{[n]}(x) \|}-\gamma_{1,2}(\A^{[n]}(x))} \\
    &\leq \frac{e^{-n(\lambda_1-\lambda_2-\frac{\varepsilon}{8} )}}{e^{-\frac{\varepsilon}{4} n}-e^{-n(\lambda_1-\lambda_2-\frac{\varepsilon}{8})}}.
\end{align*}
Hence, provided \(N\) was chosen large enough, 
\[d_{\BP}(\overline{\xi_*}(x),\overline{\upsilon_+}(\A^{[n]}(x))) \leq e^{-n(\lambda_1-\lambda_2-\varepsilon )}. \]
For the second, by Lemma \ref{lem:BQlemma}.\((i)\) we have
\[\delta_{\BP}(\overline{\upsilon_+}(\A^{[n]}(x)),\overline{v}) \geq e^{-\frac{\varepsilon}{4} n} - e^{-n(\lambda_1-\lambda_2-\frac{\varepsilon}{8})}.\]
So, provided \(N\) was chosen large enough,
\[\delta_{\BP}(\overline{\upsilon_+}(\A^{[n]}(x)),\overline{v}) \geq e^{-\frac{\varepsilon}{2}n}. \]
The third then follows as 
\begin{align*}
    \delta_{\BP}(\overline{\xi_*}(x),\overline{v}) &\geq \delta_{\BP}(\overline{\upsilon_+}(\A^{[n]}(x)),\overline{v})-d_{\BP}(\overline{\xi_*}(x),\overline{\upsilon_+}(\A^{[n]}(x)) \\
    &\geq e^{-\frac{\varepsilon}{2}n}-e^{-n(\lambda_1-\lambda_2-\varepsilon )}.
\end{align*}
Thus, provided \(N\) was chosen large enough, 
\[\delta_{\BP}(\overline{\xi_*}(x),\overline{v}) \geq e^{-\varepsilon n}. \]
\end{proof}

\begin{proof}[Proof of Proposition \ref{prop:dimensionbound}]
    This follows similarly to the proof of Theorem 14.1 in \cite{BQ16}. For any \(\varepsilon>0\), let \(\delta>0\) and \(\Lambda'<0\) be as given by Lemma \ref{lem:xi*doesntaccumulateonsubspaces} and let \(\Lambda'<\Lambda<0\). Let \(0<s<-\Lambda/\varepsilon\) and fix \(t \in [-\delta,\delta]\). By Lemma \ref{lem:xi*doesntaccumulateonsubspaces} there exists \(N \in \N\) such that for all \(n \geq N\) 
    \[\sup_{\overline{v} \in \BP} \mu_t(\{x \in \Sigma_T: \delta_{\BP}(\overline{\xi_*}(x),\overline{v}) < e^{-\varepsilon n}  \}) \leq e^{n \Lambda}.\]
    Let \(\overline{v} \in \BP\) and consider the sets 
    \[E_{n}:=\{x \in \Sigma_T:e^{-\varepsilon(n+1)} \leq \delta_{\BP}(\overline{\xi_*}(x),\overline{v})< e^{-\varepsilon n}\}.\]
     By (\ref{eqn:integratingnut}) we have
    \begin{align*}
        \int  \delta_{\BP}(\overline{u},\overline{v})^{-s} \diff \eta_t(\overline{u})&= \int  \delta_{\BP}(\overline{\xi_*}(x),\overline{v})^{-s} \diff \mu_t(x) \\
        &\leq e^{s\varepsilon N}+ \sum_{n \geq N} \int_{E_n}   \delta_{\BP}(\overline{\xi_*}(x),\overline{v})^{-s} \diff \mu_t(x) \\
        &\leq e^{s\varepsilon N}+\sum_{n \geq N} e^{\varepsilon s} e^{(\Lambda+s\varepsilon)n}<\infty.
    \end{align*}
    The bound does not depend on \(\overline{v}\) or \(t\), so 
    \[\inf_{t \in [-\delta,\delta]} \dim \eta_t \geq s>0.\]
\end{proof}

\section{Proof of Theorem \ref{theo:gibbsmeasuresexist}}\label{sec:proofoftheogibbs}
We are now ready to prove the existence statement in Theorem \ref{theo:gibbsmeasuresexist}. 

\begin{proof}[Proof of Theorem \ref{theo:gibbsmeasuresexist} (existence)]
By Proposition \ref{prop:dimensionbound} there exist \(\delta>0\) such that 
\[\inf_{s \in [-\delta,\delta]} \dim \eta_{s} > 0.\]
Hence, for all \(t<0\) sufficiently close to 0,
\[\dim \eta_t \geq \inf_{s \in [-\delta,\delta]} \dim \eta_{s} > -t. \]
For these \(t\) it follows from Lemma \ref{lem:sufficientconditionformut} that \(\mu_t\) is a Gibbs-type measure for \(\Phi_t\). By Lemmas \ref{lem:gibbsstatesareequilib}, \ref{lem:sigmainvariant}, and \ref{lem:ergodic}, the measures \(\mu_t\) are moreover ergodic equilibrium states for \(\Phi_t\). We note that by Lemma \ref{lem:sufficientconditionformut} for these \(t\) we also have
\begin{equation}\label{eqn:pressurerelationinproof}
    \Ptop(\Phi_t,\sigma)=\log \rt+\Ptop(\psi,\sigma).
\end{equation}
\end{proof}

The rest of this section is devoted to proving the uniqueness statement. Our proof of uniqueness is indirect in the following sense: we able to show that the constructed Gibbs-type measures \(\mu_t\) are unique equilibrium states for all \(t<0\) sufficiently close to 0, but not necessarily for all \(t<0\) for which the Gibbs-type measures \(\mu_t\) exist.

We will use Theorem \ref{theo:led}. Recalling Definition \ref{def:xiwelldefined}, we let
\[Y:=\{x \in \Sigma_T:\overline{\xi_*}(x) \text{ is well defined and 0-dimensional} \} \]
and note that \(\sigma^{-1}(Y)=Y\). Consider the potential \(\phi_{\A}:\Sigma_T \rightarrow \R \cup \{-\infty\}\) defined by 
\[\phi_\A(x):=\begin{cases}
    -\log \left\|\A_*^{-1}(x) \xi_*(x) \right\|, & x \in Y \\
    -\infty, & x \not\in Y.
\end{cases}\]

\begin{lemma}\label{lem:lyapunovisintegral}
    For every ergodic \(\mu \in \MCM(\Sigma_T,\sigma)\) such that \(\lambda_{1}(\A,\mu)>\lambda_2(\A,\mu)\) we have
    \[\int \phi_{\A} \diff \mu= \lambda_1(\A,\mu).\]
\end{lemma}

\begin{proof}
    Observe that for all \(x \in Y\)
\begin{equation*}\label{eqn:sumphiA}
    S_n \phi_{\A}(x) =-\log \left\|\A_*^{-n}(x) \xi_*(x) \right\|.
\end{equation*}
Thus, for each ergodic measure \(\mu \in \MCM(\Sigma_T,\sigma)\) with \(\lambda_{1}(\A,\mu)>\lambda_2(\A,\mu)\) we have 
    \(\frac{1}{n}S_n\phi_{\A}(x) \rightarrow -\lambda_d(\A_*^{-1},\mu)=\lambda_1(\A,\mu)\) for \(\mu\)-almost every \(x \in \Sigma_T\). Hence, by \(\sigma\)-invariance and the dominated convergence theorem,
    \[\int \phi_{\A} \diff \mu=\lim_{n \rightarrow \infty} \int  \frac{1}{n} S_n \phi_{\A} \diff \mu = \lambda_1(\A,\mu). \] 
\end{proof}

For \(t\) sufficiently close to 0, we can define the functions
\[g_t(x):=\frac{g(x) e^{t \phi_{\A}(x)}}{\rt } \frac{h_t(x,\overline{\xi_*}(x))}{h_t(\sigma x,\overline{\xi_*}(\sigma x))}\]
when \(x \in Y\) and 0 otherwise. 

\begin{lemma}\label{lem:gtisgfunction}
    \(g_t \in \MCG(Y)\).
\end{lemma}

\begin{proof}
For all \(x \in Y\) we have
\begin{align*}
    L_{\log g_t} 1(x) &=\sum_{y \in \sigma^{-1} x} \frac{g(y) e^{t \phi_{\A}(y)}}{\rt } \frac{ h_t(y,\overline{\xi_*}(y))}{ h_t(x,\overline{\xi_*}(x))} \\
    &= \left(h_t( x,\overline{\xi_*}(x))\right)^{-1} \rt^{-1} \sum_{y \in \sigma^{-1} x} g(y) \|\A_*^{-1}(y) \xi_*(y)\|^{-t} h_t(y,\overline{\xi_*}(y)) \\
    &= \left(h_t( x,\overline{\xi_*}(x))\right)^{-1} \rt^{-1} \sum_{y \in \sigma^{-1} x} g(y) \|\A(y)^* \xi_*(x)\|^{t} h_t(y,\overline{\A(y)^* \xi_*(x)}) \\
    &= \left(h_t( x,\overline{\xi_*}(x))\right)^{-1} \rt^{-1} \Lt h_t(x, \overline{\xi_*}(x)) \\
    &=1,
\end{align*}
where the third equality is by (\ref{eqn:inverserelation}) and the relation \(\overline{\xi_*}(y)=\overline{\A(y)^*\xi_*(\sigma y)}\). Thus, \(g_t \in \MCG(Y)\).
\end{proof}

\begin{proof}[Proof of Theorem \ref{theo:gibbsmeasuresexist} (uniqueness)]
We first show that there must exist \(\delta>0\) such that for all \(t \in (-\delta,0]\) any equilibrium state \(\mu\) of \(\Phi_t\) must have \(\lambda_1(\A,\mu)>\lambda_2(\A,\mu)\). Suppose for a contradiction that this does not hold, then we can find \(t_n \uparrow 0\) and \(\mu'_{t_n}\) such that \(\lambda_1(\A,\mu'_{t_n})=\lambda_2(\A,\mu'_{t_n})\) and
\[h(\mu_{t_n}',\sigma)+\int \psi \diff \mu'_{t_n}+t_n \lambda_1(\A,\mu'_{t_n})=\Ptop(\Phi_{t_n},\sigma).\]
By (\ref{eqn:pressurerelationinproof}), Proposition \ref{prop:derivativeislyapunov}, and the variational principle, we must have \(\lambda_1(\A,\mu'_{t_n})=\lambda_1(\A,\mu_{t_n})\) for all \(n\) sufficiently large, otherwise neither the pressure function nor \(t \mapsto \log \rt\) would be differentiable at \(t_n\).

Let \(\mu' \in \MCM(\Sigma,\sigma)\) be a weak* limit point of the sequence \((\mu'_{t_n})_{n \in \N}\). By upper semi-continuity of the entropy map, 
\begin{align*}
    \Ptop(\psi,\sigma)&=\lim_{n \rightarrow \infty}\Ptop(\Phi_{t_n},\sigma) \\
    &=\lim_{n \rightarrow \infty} h(\mu_{t_n}',\sigma)+\int \psi \diff \mu'_{t_n}+t_n \lambda_1(\A,\mu'_{t_n})\\
    &\leq h(\mu',\sigma)+\int \psi \diff \mu'.
\end{align*}
Hence, \(\mu'=\mu_{\psi}\), where \(\mu_{\psi}\) is the Gibbs measure for \(\psi\). Passing to a subsequence if necessary, we may assume that \(\mu'_{t_n} \rightarrow \mu_{\psi}\) weak*. By \cite[Proposition A.1.(2)]{FH10}, we have that \(\mu \mapsto 
\lambda_1(\A,\mu)+\lambda_2(\A,\mu)\) is upper semi-continuous in the weak* topology, which implies 
\[\lambda_1(\A,\mu_{\psi})+\lambda_2(\A,\mu_{\psi}) \geq \limsup_{n \rightarrow \infty} \lambda_1(\A,\mu'_{t_n})+\lambda_2(\A,\mu'_{t_n}) = \limsup_{n \rightarrow \infty} 2\lambda_1(\A,\mu'_{t_n})= 2\lambda_1(\A,\mu_{\psi}),\]
so \(\lambda_1(\A,\mu_{\psi})=\lambda_2(\A,\mu_{\psi})\). However, this contradicts \cite[Theorem 1]{BV04}.

Thus, by considering \(t\) to be closer to 0 if necessary, we may assume that any equilibrium state of \(\Phi_t\) must have \(\lambda_1(\A,\mu)>\lambda_2(\A,\mu)\). For \(t<0\) sufficiently close to 0, we now show that the Gibbs measures \(\mu_t\) constructed before are unique equilibrium states.

Fix \(t<0\) sufficiently close to 0 and let \(\mu\) be an ergodic equilibrium measure for \(\Phi_t\). We will show that \(\mu=\mu_t\). Note that it then follows immediately that \(\mu_t\) is the unique equilibrium state, since each equilibrium measure is a barycentre of the collection of ergodic equilibrium measures by \cite[Proposition A.1.(3)]{FH10} and \cite[Theorem 8.4]{Wal81}. 

From the definition of \(g_t\) and by Lemma \ref{lem:lyapunovisintegral} we have
\[h(\mu,\sigma)+\int \log g_t \diff \mu=h(\mu,\sigma)+\int \log g \diff \mu+t \lambda_1(\A,\mu)-\Ptop(\Phi^g_t,\sigma)= 0.\]
Hence, by Theorem \ref{theo:led} we have that 
\[L_{\log g_t}^* \mu=\mu.\]

Let \(k \in \N\) be such that for all \(I,J \in \MCC_*\) there exists \(K \in \MCC_k\) such that \(IKJ\) is admissible. Let \(n \in \N\) and \(I \in \MCC_n\). For each \(x \in \Sigma_T\) there is at least one \(z \in \sigma^{-(n+k)}(x)\) with \(z \in [I]\). Hence,
\begin{align*}
    L_{\log g_t}^{n+k} 1_{[I]}(x) &= \rt^{-(n+k)} \sum_{y \in \sigma^{-(n+k)} x} g^{(n+k)}(y) \|\A_*^{-(n+k)}(y)\xi_*(y)\|^{-t} 1_{[I]}(y) \frac{h_t(y,\overline{\xi_*}(y))}{h_t( x,\overline{\xi_*}(x))} \\
    &\geq \rt^{-(n+k)} (\inf h_t )(\sup h_t)^{-1} g^{(n+k)}(z)  \|\A_*^{-(n+k)}(z)\xi_*(z)\|^{-t} \\
    &=  \rt^{-(n+k)} (\inf h_t )(\sup h_t)^{-1} g^{(n+k)}(z)  \|\A^{[n+k]}(z)\xi_*(x)\|^{t} \\
    &\geq  \rt^{-(n+k)} (\inf h_t )(\sup h_t)^{-1}g^{(n+k)}(z)  \|\A^{[n+k]}(z)\|^{t} \\
    &\geq \left(\rt^{-{k}} e^{-k \|\log g\|_{\infty}} \left(\max_{y \in \Sigma_T}\|\A^k(y)\| \right)^t(\inf h_t )(\sup h_t)^{-1} \right) \rt^{-n} g^{(n)}(z)  \|\A^{[n]}( z)\|^{t} \\
    &= \left(\rt^{-{k}} e^{-k \|\log g\|_{\infty}}  \left(\max_{y \in \Sigma_T}\|\A^k(y)\| \right)^t (\inf h_t )(\sup h_t)^{-1} \right) e^{-n \Ptop(\Phi^g_t,\sigma)} g^{(n)}(z) \|\A^{n}(I)\|^{t},
\end{align*}
where the second equality is by (\ref{eqn:inverserelation}). Thus, by bounded distortion there exists \(C=C(t)>0\) such that for all \(n \in \N\) and \(I \in \MCC_n\),
\[\mu([I]) \geq C e^{-n \Ptop(\Phi^g_t,\sigma)} \|\A^{n}(I)\|^{t} \inf_{x \in [I]} g^{(n)}(x).\]
Since \(\mu_t\) is an ergodic Gibbs-type measure for \(\Phi_t^g\), by Lemma \ref{lem:otherequil} we must have \(\mu=\mu_t\).
\end{proof}

\section{Proof of Theorem \ref{theo:Ptopanalytic}}\label{sec:proofofanalytictheorem}    
By (\ref{eqn:pressurerelationinproof}) we have that \(\Ptop(\Phi_t,\sigma)=\log \rho_t+\Ptop(\psi,\sigma)\) for all \(t<0\) sufficiently small. Since \(t \mapsto \log \rho_t\) is analytic in a neighbourhood of 0 with derivative \(\lambda_1(\A,\mu_t)\) at \(t\), to prove Theorem \ref{theo:Ptopanalytic} it suffices to show that \(\Ptop(\Phi_t,\sigma)=\log \rho_t+\Ptop(\psi,\sigma)\) for \(t> 0\) sufficiently small. Hence, to finish the proof of Theorem \ref{theo:Ptopanalytic}, we only need to prove the following lemma.

\begin{lemma}\label{lem:mutgibbst>0}
    For all \(t> 0\) sufficiently close to 0, \(\mu_t\) is Gibbs-type for \(\Phi_t\) and \(\Ptop(\Phi_t,\sigma)=\log \rho_t+\Ptop(\psi,\sigma)\).  
\end{lemma}
\begin{proof}
 We adapt the proof of Proposition 3.4 in \cite{Pir20}. It is a consequence of Proposition \ref{prop:dimensionbound} that, for all \(t> 0\) small enough and all \(i\) in the alphabet \(\{1,\ldots, q\}\), the measures \(\eta_t^{(i)}\) defined by 
    \[\int f(\overline{u}) \diff \eta_t^{(i)}= \int 1_{[i]}(x) f(\overline{u}) \diff \nu_t\]
    are not supported on a projective subspace. By Lemma \ref{lem:upperboundformutJ} they are not the zero measure either. It follows that for each \(1\leq i \leq q\),
    \[A \mapsto \int \left\|A \uou \right\|^t \diff \eta_t^{(i)}\]
    is strictly positive on the set of norm one \(d \times d\) matrices with entries in \(\R\). Since these maps are also continuous, there exists \(C>0\) such that for all \(A \in M_d(\R)\) and all \(1 \leq i \leq q\),
    \[\int \left\|A \uou \right\|^t \diff \eta_t^{(i)} \geq C \|A\|^t. \]
    By Lemmas \ref{lem:boundeddistortion}, \ref{lem:htbound}, and  \ref{lem:mutJequality}, we have that for any \(n \in \N\), \(I \in \MCC_n\), and \(y \in [I]\),
\begin{align*}
    \frac{\mu_t([I])}{\rt^{-n}}
    &=\int g^{(n)}(Ix) \left\|\A^{[n]}(I)\uou \right\|^t  1_{\{ x: (I,x_0) \text{ admissible}\}}(x) h_t(Ix, \overline{\A^{[n]}(I)u}) \diff \widetilde{\nu}_t(x,\overline{u}) \\ 
    &\geq C_{b.d} g^{(n)}(y) \inf(h_t) \sup(h_t)^{-1} \int  \|\A^{[n]}(y)u\|^t 1_{\{ x: (I,x_0) \text{ admissible}\}}(x) \diff \nu_t(x,\overline{u}) \\
    &\geq C_{b.d} g^{(n)}(y) \inf(h_t) \sup(h_t)^{-1} \min_{1\leq i \leq q} \left\{ \int  \left\|\A^{[n]}(y) \uou \right\|^t 1_{[i]}(x) \diff \nu_t(x,\overline{u}) \right\} \\
    &\geq C C_{b.d} g^{(n)}(y) \inf(h_t) \sup(h_t)^{-1} \|\A^{[n]}(y)\| \\
    &= C C_{b.d} g^{(n)}(y) \inf(h_t) \sup(h_t)^{-1} \|\A^{n}(y)\|.
\end{align*}
Thus, there exists \(C_1=C_1(t)>0\) such that for all \(n \in \N\), \(I \in \MCC_n\), and \(y \in [I]\),
\[C_1 \leq \frac{\mu_t([I])}{\rt^{-n} g^{(n)}(y) \|\A^{n}(y)\|^t}.\]

The other direction can be shown straightforwardly using the trivial inequality \(\|A u\| \leq \|A\| \) for all \(A \in \GL_d(\R)\) and all unit \(u \in \R^d\). Thus, for \(t >0\) sufficiently small, \(\mu_t\) is a Gibbs-type measure for \(\Phi^g_t\) and hence for \(\Phi_t\).  By Lemma \ref{lem:PequalsPtop} (noting that the proof also holds for subadditive sequences) and equation (\ref{eqn:PtopPhig}), we further have that \(\Ptop(\Phi_t,\sigma)=\log \rt+\Ptop(\psi,\sigma)\). This concludes the proof of the lemma and, as argued above, of Theorem \ref{theo:Ptopanalytic}.
\end{proof}

\begin{remark}
    For both positive and negative \(t\) close to 0, rather than showing that the measures \(\mu_t\) have the Gibbs-type property, an alternative way of showing they are equilibrium states for \(\Phi_t\) is to instead use Theorem \ref{theo:led} in conjunction with Proposition \ref{prop:nutsupportonxi*} and Lemmas \ref{lem:lyapunovisintegral} and \ref{lem:gtisgfunction}. In this way, we expect it should be possible to prove Theorem \ref{theo:Ptopanalytic} under the more general fiber-bunching condition.
\end{remark}

\section{Convexity of the pressure}\label{sec:convexity}
It is an easy consequence of the variational principle that the pressure function \(t \mapsto \Ptop(\Phi,t)\) is convex on \(\R\). We have the following application of Theorem \ref{theo:Ptopanalytic} regarding strict convexity. Let \(\mu_{\psi} \in \MCM(\Sigma_T,\sigma)\) be the Gibbs measure corresponding to the H\"older continuous potential \(\psi\). 

\begin{theo}\label{theo:convexity}
    In the setting of Theorems \ref{theo:gibbsmeasuresexist} and \ref{theo:Ptopanalytic}, the following are equivalent:
    \begin{enumerate}[label={(\arabic*)}]
        \item \(t \mapsto \Ptop(\Phi_t,\sigma)\) is not strictly convex in a neighbourhood of 0,
        \item there exists \(C>0\) such that for all \(x \in \Sigma_T\) and \(n \in \N\)
        \[\left|\frac{1}{n} \log \|\A^n(x)\|-\lambda_1(\A,\mu_{\psi}) \right|<C/n, \]
        \item \(\lim_{n \rightarrow \infty} \frac{1}{n}\log \|\A^n(x)\|=\lambda_1(\A,\mu_{\psi})\) for all \(x \in \Sigma_T\),
        \item \(\lambda_1(\A,\mu)=\lambda_1(\A,\mu_{\psi})\) for all \(\mu \in \MCM(\Sigma_T,\sigma)\),
        \item \(\Ptop(\Phi_t,\sigma)=\Ptop(\psi,\sigma)+t \lambda_1(\A,\mu_{\psi})\) for all \(t \in \R\).
    \end{enumerate}
\end{theo}

\begin{proof}
    We first show \((1) \implies (2)\). Since \(t \mapsto \Ptop(\Phi_t,\sigma)\) is analytic in a neighbourhood of 0 with derivative \(\lambda_1(\A,\mu_{\psi})\) at 0, if it is not strictly convex in a neighbourhood of 0 then \begin{equation}\label{eqn:tfaelemma}
    \Ptop(\Phi_t,\sigma)=\Ptop(\psi,\sigma)+t \lambda_1(\A,\mu_{\psi})
    \end{equation}
    for all \(t \in [-\delta,\delta]\) and some \(\delta>0\). By the variational principle, this implies that \(\mu_{\psi}\) is an equilibrium state for \(\Ptop(\Phi_t,\sigma)\) for these \(t \in [-\delta,\delta]\). Moreover, the equilibrium states for quasi-additive subadditive potentials with bounded variations are unique and are Gibbs-type states (see \cite[Theorem 5.5]{Fen11}), so \(\mu_{\psi}\) is the Gibbs-type state for \(\Ptop(\Phi_t,\sigma)\) for all \(t \in [0,\delta]\) by \cite{Par20}. Thus, there exists \(C>1\) such that for all \(x \in \Sigma_T\) and \(n \in \N\),
    \[\frac{1}{C} \leq \frac{\mu_{\psi}([x_0,\ldots,x_{n-1}])}{e^{-n \Ptop(\psi,\sigma)+S_n \psi(x)} } \leq C \]
    and 
    \[\frac{1}{C} \leq \frac{\mu_{\psi}([x_0,\ldots,x_{n-1}])}{e^{-n \Ptop(\Phi_\delta,\sigma)+S_n \psi(x)}\|\A^n(x)\|^\delta } \leq C.\]
    By (\ref{eqn:tfaelemma}) this implies that for all \(x \in \Sigma_T\) and \(n \in \N\), 
    \[\frac{1}{C^2} \leq \frac{e^{n \delta \lambda_1(\A,\mu_{\psi})}}{\|\A^n(x)\|^\delta} \leq C^2,\]
    so 
    \[\left|\frac{1}{n} \log \|\A^n(x)\|-\lambda_1(\A,\mu_{\psi}) \right|<\frac{2\log C}{\delta n}.\]
    In particular, \((1) \implies (2)\). Notice that \((2) \implies (3)\), \((3) \implies (4)\) and \((5) \implies (1)\) are trivial. Furthermore, \((4) \implies (5)\) by the variational principle for \(\Phi_t\) which holds for all \(t \in \R\).
    \end{proof}

    Theorem \ref{theo:convexity} has applications towards large deviations. Theorem B in \cite{PPi22} assumes that \(\rho''_0=\Var(\A,\mu_{\psi})>0\), but the proof only uses that \(\log \rt\) is strictly convex in a neighbourhood of 0. Thus, when \(\A\) is one-step, the assumption \(\Var(\A,\mu_{\psi})>0\) can be relaxed to instead assume that (2) in Theorem \ref{theo:convexity} is not satisfied. We also remark that the rate function in their theorem is given by the Legendre transform of \(\Lambda(\varepsilon):=\log \rt-\varepsilon \lambda_1(\A,\mu_{\psi})\), which by the proof in our paper can be written in terms of the topological pressure in a neighbourhood of 0 (again in the one-step setting).

\section{Examples with a phase transition when \(t<0\)}\label{sec:proofoftheocounter}
Using results proved in \cite{DGR19,DGR22}, we now show that Theorems \ref{theo:gibbsmeasuresexist} and \ref{theo:Ptopanalytic} cannot be extended to all \(t<0\). We do this by showing that there are examples in the setting of this paper where the pressure function has a phase transition (a point of non-analyticity) at some point \(t<0\). Examples of this phenomenon have been shown to occur previously in \cite{Fen09} (Example 6.6) under a different irreducibility condition which allows non-invertible matrices. In their setting they show the pressure function can fail to be differentiable, but it is not known whether this is possible in our setting. 

In the papers \cite{DGR19,DGR22} they consider a subset of one-step \(\SL(2,\R)\)-cocycles which they call \textit{elliptic cocycles with some hyperbolicity} -- for the definition see \cite[\S 11.7]{DGR19}. For simplicity we focus on a specific example but remark that the arguments apply to their general setting. Note that it is immediate from the definition that the cocycles in the following proposition are elliptic cocycles with some hyperbolicity. It is also easy to check that they are 1-typical. 

\begin{prop}\label{prop:Ptopnotanalytic}
    For \(\lambda>1\) and \(\theta \in \R \setminus \Q\), let
\[A_1:=\begin{pmatrix}
        \lambda & 0 \\
        0& 1/\lambda \end{pmatrix}, A_2:= \begin{pmatrix}
        \cos (2 \pi \theta) & -\sin (2 \pi \theta) \\
         \sin (2 \pi \theta) & \cos (2 \pi \theta)
    \end{pmatrix}\]
Also, let \((\Sigma,\sigma)\) be the full shift on alphabet \(\{1,2\}\). For \(t \in \R\) and \(n \in \N\) define \(\varphi_{t,n}:\Sigma \rightarrow \R\) by
    \[\varphi_{t,n}(x):=t\log \|\A^n(x)\|,\]
where \(\A^n(x):=A_{x_{n-1}} \ldots A_{x_0}\), and let \(\Phi_t:=(\varphi_{t,n})_{n \in \N}\). Then, there exists \(C>0\) such that
\begin{equation}\label{eqn:tlessthan-Cinprop}
    \Ptop(\Phi_t,\sigma)=\sup_{\mu \in \MCM(\Sigma,\sigma)}\{h(\mu,\sigma): \lambda_1(\A, \mu)=0\}, \: \: \forall t \leq -C,
\end{equation}
    and 
\begin{equation}\label{eqn:tgreaterthan-Cinprop}
    \Ptop(\Phi_t,\sigma)>\sup_{\mu \in \MCM(\Sigma,\sigma)} \{h(\mu, \sigma): \lambda_1(\A,\mu)=0\}, \: \: \forall t>-C.
\end{equation}
In particular, \(t \mapsto \Ptop(\Phi_t,\sigma)\) is not analytic.
\end{prop}

Before we prove Proposition \ref{prop:Ptopnotanalytic}, let us first introduce some notation. We fix an \(\A\) as in Proposition \ref{prop:Ptopnotanalytic}. For a point \(x \in \Sigma\) we denote
\[\lambda_1(x):=\lim_{n \rightarrow \infty} \frac{1}{n} \log \|\A^n(x)\|,\]
when the limit exists. It is clear that \(\lambda_1(x) \geq 0\) for all \(x \in \Sigma\) for which this is well defined. For \(\alpha \geq 0\), let 
\[X(\alpha):=\{x \in \Sigma: \lambda_1(x)=\alpha\}\]
and define 
\[\alpha_{\max} := \sup \{\alpha \in \R : X(\alpha)\not=\emptyset\}.\]
The maximum is attained, justifying the notation. Recall the definition of topological entropy \(\htop\) defined by Bowen in \cite{Bow73} which also applies to non-compact sets (see also \cite[Appendix B]{DGR19}). By Theorem B in \cite{DGR19}, for all \(\alpha \in (0, \alpha_{\max}]\),
\begin{equation}\label{eqn:DGRresult}
    \htop(X(\alpha),\sigma)= \sup_{\mu \in \MCM_{\text{erg}}(\Sigma,\sigma)} \{h(\mu,\sigma): \lambda_1(\A,\mu)=\alpha \},
\end{equation}
where \(\MCM_{\text{erg}}(\Sigma,\sigma)\) is the subset of \(\MCM(\Sigma,\sigma)\) consisting of the ergodic measures. Theorem B in \cite{DGR22} says that this further holds for \(\alpha=0\). That is, 
\begin{align}
        \htop(X(0),\sigma)&= \sup_{\mu \in \MCM_{\text{erg}}(\Sigma,\sigma)} \{h(\mu, \sigma): \lambda_1(\A,\mu)=0 \}\nonumber \\
        &=\sup_{\mu \in \MCM(\Sigma,\sigma)} \{h(\mu, \sigma): \lambda_1(\A,\mu)=0 \}, \label{eqn:DGRresult0}
\end{align}
where the second equality follows from \cite[Proposition A.1.(3)]{FH10} and \cite[Theorem 8.4]{Wal81}.

\begin{proof}[Proof of Proposition \ref{prop:Ptopnotanalytic}]
Combining Theorem 4.(i) and Theorem 5 in \cite{DGR19}, the right derivative at 0 of \(\alpha \mapsto \htop(X(\alpha),\sigma)\) is equal to some \(0<C<\infty\). Also, Theorem 4.(g) and Theorem 5 in \cite{DGR19} imply that \(\alpha \mapsto \htop(X(\alpha),\sigma)\) is convex on \([0,\alpha_{\max}]\), so it follows that for all \(\alpha \in [0,\alpha_{\max}]\),
    \[\htop(X(\alpha),\sigma) \leq \htop(X(0),\sigma) +C \alpha.\]

Fix some \(t \leq -C\). We first show (\ref{eqn:tlessthan-Cinprop}). By (\ref{eqn:DGRresult}), for all \(\mu \in \MCM_{\text{erg}}(\Sigma,\sigma)\),
\begin{align*}
    h(\mu,\sigma)+t \lambda_1(\A,\mu) &\leq \htop(X(\lambda_1(\A,\mu)),\sigma)+ t\lambda_1(\A,\mu) \\ 
    &\leq \htop(X(0),\sigma) +C \lambda_1(\A,\mu) + t\lambda_1(\A,\mu) \\
    &\leq \htop(X(0),\sigma).
\end{align*}
This further holds for all \(\mu \in \MCM(\Sigma,\sigma)\) by \cite[Proposition A.1.(3)]{FH10} and \cite[Theorem 8.4]{Wal81}, so 
\[\Ptop(\Phi_t,\sigma)=\Pmeas(\Phi_t,\sigma) \leq \htop(X(0),\sigma).\]

Now observe that, by (\ref{eqn:DGRresult0}), for all \(\varepsilon>0\) there exists \(\mu \in \MCM(\Sigma,\sigma)\) with \(\lambda_1(\A,\mu)=0\) such that
\[h(\mu, \sigma)+t \lambda_1(\A,\mu) =h(\mu,\sigma)> \htop(X(0),\sigma)-\varepsilon.\]
It therefore follows that 
\begin{equation}\label{eqn:ptopislinear}
    \Ptop(\Phi_t,\sigma)=\Pmeas(\Phi_t,\sigma) \geq \htop(X(0),\sigma).
\end{equation}
Hence, we have shown that 
\[\Ptop(\Phi_t,\sigma)=\htop(X(0),\sigma),\]
so, by (\ref{eqn:DGRresult0}), statement (\ref{eqn:tlessthan-Cinprop}) follows.

We now prove (\ref{eqn:tgreaterthan-Cinprop}). Fix \(t>-C\) and let \(0<\varepsilon<t+C\). Since the right derivative of \(\alpha \mapsto \htop(X(\alpha),\sigma)\) at 0 is equal to \(C\), for all \(\alpha>0\) sufficiently small we have that
\[\htop(X(\alpha),\sigma) \geq \htop(X(0),\sigma) +(C-\varepsilon) \alpha.\]
For \(\alpha>0\) sufficiently small, by (\ref{eqn:DGRresult}) we can let \(\mu \in \MCM_{\text{erg}}(\Sigma,\sigma)\) be such that \(\lambda_1(\A,\mu)=\alpha\) and
\[h(\mu,\sigma) \geq \htop(X(\alpha),\sigma)-\frac{(C+t-\varepsilon)\alpha}{2}.\]
It follows that 
\[h(\mu,\sigma)+t\lambda_1(\A,\mu) \geq \htop(X(0),\sigma)+\frac{(C+t-\varepsilon)\alpha}{2}> \htop(X(0),\sigma),  \]
so \(\Ptop(\Phi_t,\sigma)>\htop(X(0),\sigma)\). Hence, by (\ref{eqn:DGRresult0}), statement (\ref{eqn:tgreaterthan-Cinprop}) follows.
\end{proof}

\begin{remark}
    In \cite[Example 6.6]{Fen09} they show that the pressure function in their setting can fail to be differentiable, whereas it only follows from Proposition \ref{prop:Ptopnotanalytic} that analyticity can fail. We remark that the pressure not being differentiable at \(t=-C\) above would imply that \(\alpha \mapsto \htop(X(\alpha),\sigma)\) is linear on \([0,\epsilon]\) for some \(\epsilon>0\).
\end{remark}

To show that the existence of Gibbs-type measures proved in Theorem \ref{theo:gibbsmeasuresexist} cannot be extended to all \(t<0\), one could use Proposition \ref{prop:Ptopnotanalytic} and an argument similar to the one used in Theorem \ref{theo:convexity}. Instead, we finish by giving an example which provides an elementary proof of this, without using the results proved in \cite{DGR19,DGR22}. It is straightforward to check that the cocycle defined in the following proposition is 1-typical (it is also an elliptic cocycle with some hyperbolicity).

\begin{prop}\label{theo:counterexample}
    For \(\lambda>1\) and \(\theta \in \R \setminus \Q\), let
    \[A_1:=\begin{pmatrix}
        \lambda & 0 \\
        0& 1/\lambda
    \end{pmatrix}, \: A_2:=\begin{pmatrix}
        0 & 1 \\
        1& 0
    \end{pmatrix}, \: A_3:= \begin{pmatrix}
        \cos (2 \pi \theta) & -\sin (2 \pi \theta) \\
         \sin (2 \pi \theta) & \cos (2 \pi \theta)
    \end{pmatrix}.\]
    Also, let \((\Sigma,\sigma)\) be the full shift on alphabet \(\{1,2,3\}\). For \(t \in \R\) and \(n \in \N\) define \(\varphi_{t,n}:\Sigma \rightarrow \R\) by
    \[\varphi_{t,n}(x):=t\log \|\A^n(x)\|,\]
    where \(\A^n(x):=A_{x_{n-1}} \ldots A_{x_0}\). Then, for all \(t< \frac{-\log(9/2)}{\log \lambda}\) there does not exist a Gibbs-type measure for \(\Phi_t:=(\varphi_{t,n})_{n \in \N}\).
\end{prop}

\begin{proof}[Proof of Proposition \ref{theo:counterexample}]
Since \(A_2,A_3 \in O_2(\R)\), the Bernoulli measure with weights \((0,1/2,1/2)\) has top Lyapunov exponent equal to 0. Hence, by the variational principle \(\Ptop(\Phi_t,\sigma) \geq \log 2\) for all \(t<0\). On the other hand, because \(A_1,A_2,A_3\) are all elements of \(\mathrm{SL}_2(\R)\) we have that \(\lambda_1(\A,\mu) \geq 0\) for every \(\mu \in \MCM(\Sigma,\sigma)\). So, again by the variational principle, \( \Ptop(\Phi_t,\sigma) \leq \log 3\) for all \(t<0\). We remark that the inequality 
\[\log 2 \leq \Ptop(\Phi_t,\sigma) \leq \log 3, \: \forall t<0\]
can also be easily seen directly from the definition of topological pressure.

Fix \(t<\frac{-\log(9/2)}{\log \lambda}\) and suppose for a contradiction that there exists a Gibbs-type measure \(\mu\) for \(\Phi_t\). By the Gibbs-type property there exists \(C>0\) such that for all \(n \in \N\) and all \(I \in \{1,2,3\}^n\), 
\[\frac{1}{C} \leq \frac{\mu([I])}{e^{-n\Ptop(\Phi_t)} \|\A^n(I)\|^t}\leq C.\]
Choose \(n \in \N\) large enough so that 
\[\left(\frac{2 \lambda^{-t}}{9} \right)^n >3C^2,   \]
and let \[I:=(\underbrace{1,\ldots, 1}_{n \text{ many 1s} }), \: \: J:=(\underbrace{1,\ldots, 1}_{n \text{ many 1s} },2, \underbrace{1,\ldots, 1}_{n \text{ many 1s} }).\]
We have
\begin{align*}
\mu([I]) &\leq C e^{-n \Ptop(\Phi_t,\sigma)} \|\A^n(I)\|^t \\
&= C e^{-n \Ptop(\Phi_t,\sigma)} \lambda^{nt} \\
&\leq C 2^{-n} \lambda^{nt}
\end{align*}
and
\begin{align*}
    \mu([J]) &\geq C^{-1} e^{-(2n+1) \Ptop(\Phi_t,\sigma)} \|\A^{2n+1}(J)\|^t \\
    &= C^{-1} e^{-(2n+1) \Ptop(\Phi_t,\sigma)} \\
    & \geq C^{-1} 3^{-(2n+1)}.
\end{align*}
Thus, by our choice of \(n\)
\[\mu([J]) \geq C^{-1}3^{-(2n+1)} >  C 2^{-n} \lambda^{nt} \geq \mu([I]),\]
which is a contradiction since \([J] \subseteq [I]\).
\end{proof}

\subsection*{Acknowledgements}
I am very grateful to Thomas Jordan, my PhD supervisor, for numerous discussions on this topic, for reading several drafts of this paper, and for his subsequent advice. I would like to thank Henna Koivusalo, Jonathon Fraser, Mark Pollicott, and the anonymous referee for their many helpful comments and suggestions. I would also like to thank Alex Rutar for some informative discussions early on in this project. 

\section*{Declarations}
\subsection*{Funding}
This work was supported by an EPSRC DTP at the University of Bristol, studentship 2278542. 

\subsection*{Conflicts of interest}
The author has no conflicts of interest to declare that are relevant to the content of this article.

\end{document}